\documentclass[11pt,a4paper]{article}
\usepackage[left=2.0cm, top=2.45cm,bottom=2.45cm,right=2.0cm]{geometry}
\usepackage{mathtools,amssymb,amsthm,mathrsfs,calc,graphicx,xcolor,cleveref,dsfont,tikz,pgfplots,bm}
\usepackage[british]{babel}
\usepackage{amsfonts}              
\usepackage[T1]{fontenc}
\numberwithin{equation}{section}
\numberwithin{figure}{section}

\newtheorem {theorem}{Theorem}[section]
\newtheorem {proposition}[theorem]{Proposition}
\newtheorem {lemma}[theorem]{Lemma}
\newtheorem {corollary}[theorem]{Corollary}

{\theoremstyle{definition}

}
{\theoremstyle{theorem}
\newtheorem {remark}[theorem]{Remark}

}

\newcommand{\diam}{\operatorname{diam}}
\newcommand{\dint}{\textup{d}}

\def\EE{\mathbb{E}}

\def\MM{\mathbb{M}}
\def\NN{\mathbb{N}}
\def\OO{\mathbb{O}}
\def\PP{\mathbb{P}}
\def\QQ{\mathbb{Q}}
\def\RR{\mathbb{R}}
\def\SS{\mathbb{S}}

\def\XX{\mathbb{X}}

\def\bN{\mathbf{N}}

\def\cA{\mathcal{A}}

\def\cC{\mathcal{C}}

\def\cH{\mathcal{H}}

\def\cN{\mathcal{N}}

\def\cX{\mathcal{X}}

\begin{document}

\title{\bfseries Concentration inequalities for \\ functionals of Poisson cylinder processes}

\author{Anastas Baci\footnotemark[1],\; Carina Betken\footnotemark[2],\; Anna Gusakova\footnotemark[3]\;\; and Christoph Th\"ale\footnotemark[4]}

\date{}
\renewcommand{\thefootnote}{\fnsymbol{footnote}}
\footnotetext[1]{Ruhr University Bochum, Germany. Email: anastas.baci@rub.de}

\footnotetext[2]{Ruhr University Bochum, Germany. Email: carina.betken@rub.de}

\footnotetext[3]{Ruhr University Bochum, Germany. Email: anna.gusakova@rub.de}

\footnotetext[4]{Ruhr University Bochum, Germany. Email: christoph.thaele@rub.de}

\maketitle

\begin{abstract}
\noindent  Random union sets $Z$ associated with stationary Poisson processes of $k$-cylinders in $\mathbb{R}^d$ are considered. Under general conditions on the typical cylinder base a concentration inequality for the volume of $Z$ restricted to a compact window is derived. Assuming convexity of the typical cylinder base and isotropy of $Z$ a concentration inequality for intrinsic volumes of arbitrary order is established. A number of special cases are discussed, for example the case when the cylinder bases arise from a random rotation of a fixed convex body. Also the situation of expanding windows is studied. Special attention is payed to the case $k=0$, which corresponds to the classical Boolean model.
\end{abstract}


\section{Introduction}

The stationary Boolean model is one of the most versatile models considered in stochastic geometry. Its numerous applications range, for example, from coverage optimization in telecommunication networks to questions related to virtual material design. While mean value formulas for the intrinsic volumes of the Boolean model are rather classical (see, e.g., \cite{SW}), only recently a satisfactory description of second-order properties was derived by Hug, Last and Schulte \cite{HugLastSchulte} together with an accompanying central limit theory. In addition, building on a concentration inequality for Poisson functionals on abstract phase spaces Gieringer and Last \cite{GieringerLast} obtained concentration inequalities for a class of measures associated with a rather general Boolean model in an observation window. On their way they were able to refine earlier estimates of Heinrich \cite{HeinrichBM} for the volume of a stationary Boolean model in $\RR^d$ restricted to a compact observation window, which in turn were obtained by means of sharp bounds on cumulants.

The aim of the present paper is to prove concentration inequalities for the volume as well as for the intrinsic volumes associated with the union set of a stationary Poisson cylinder process in $\RR^d$ restricted to a compact window. For $k\in\{0,1,\ldots,d-1\}$ we understand by a $k$-cylinder the Minkowski sum of a $k$-dimensional linear subspace in $\RR^d$ and a compact set in its orthogonal complement. A Poisson process of $k$-cylinders (or Poisson cylinder process for short) is a Poisson process on the space of $k$-cylinders in $\RR^d$. We refer to Section \ref{sec:PCP} for a formal description of the model. In this paper we consider the union set $Z$ associated with such a Poisson cylinder process, which is observed in a compact window $W\subset\RR^d$. It is assumed throughout that $Z$ is a stationary random closed set. In this case the distribution of $Z$ is determined by an intensity parameter $\gamma\in(0,\infty)$ as well as the distribution $\QQ$ of the pair $(\Xi,\Theta)$, where $\Xi$ describes the base and $\Theta$ the direction of the typical cylinder. It is worth pointing out that the concept of a Poisson cylinder process generalizes that of the Boolean model discussed above, which is included as the special case $k=0$. In this situation $\Xi$ is the typical grain of the Boolean model and the random direction $\Theta$ has no relevance. Poisson cylinder processes were formally introduced by Matheron \cite{Materon}, Miles \cite{Miles} and Weil \cite{Weil}. More recently, central limit theorems for stationary Poisson processes of cylinders were studied by Heinrich and Spiess \cite{HeinrichSpiessCLTVolume,HeinrichSpiessCLTVundS}. Under an exponential moment assumption on the $(d-k)$-volume of the typical cylinder base they obtained in \cite{HeinrichSpiessCLTVolume} a central limit theorem for the volume of $Z$ in a sequence of growing windows, that is, for $Z\cap W_r$, where $W_r=rW$, as $r\to\infty$. More precisely, using sharp bounds on cumulants they were able to deduce a rate of convergence as well as Cram\'er-type large deviations. In a subsequent paper \cite{HeinrichSpiessCLTVundS} they were able to relax the moment assumptions and to add a central limit theorem for the surface content. Characteristic quantities like volume fraction, covariance function and contact distribution functions of anisotropic Poisson cylinder processes were investigated by Spiess and Spodarev \cite{SpiessSpodarev}. In addition, percolation and connectivity properties related to Poisson cylinder processes with spherical bases and $k=1$ were studied by Tykesson and Windisch \cite{TykessonWindisch}, Hilario, Sidoravicius and Teixeira \cite{Hilario} as well as Borman and Tykesson \cite{BormanTykesson}.

\begin{figure}[t]
    \centering
    \includegraphics[width=0.4\columnwidth]{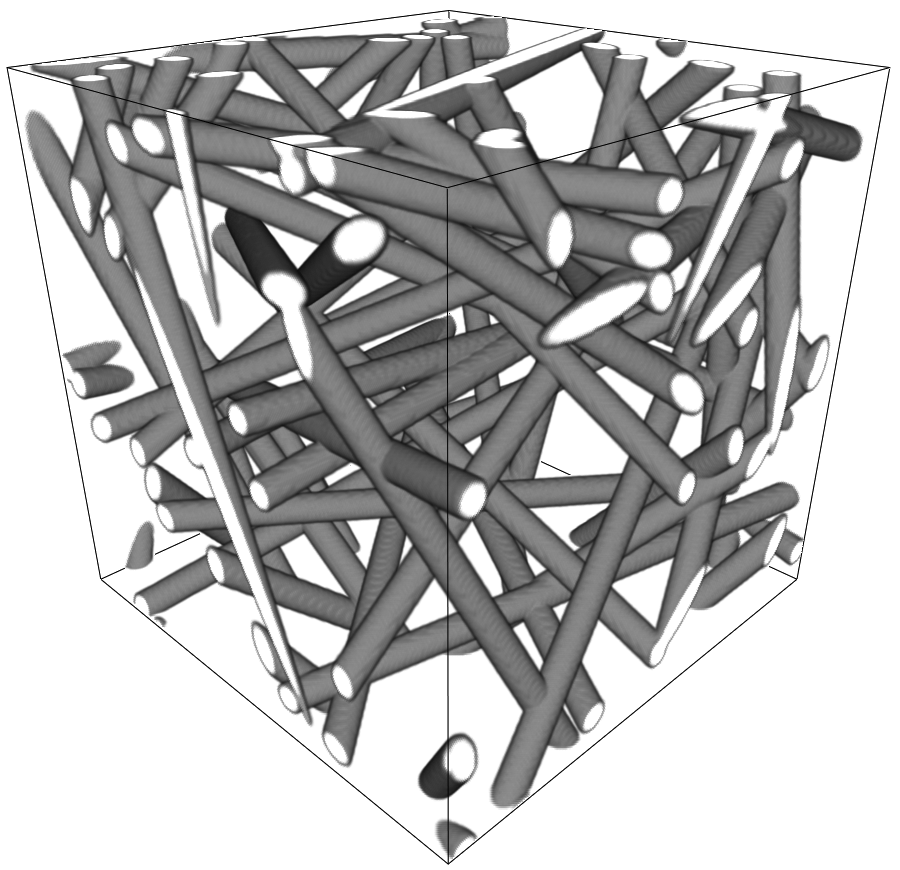}\quad
    \includegraphics[width=0.4\columnwidth]{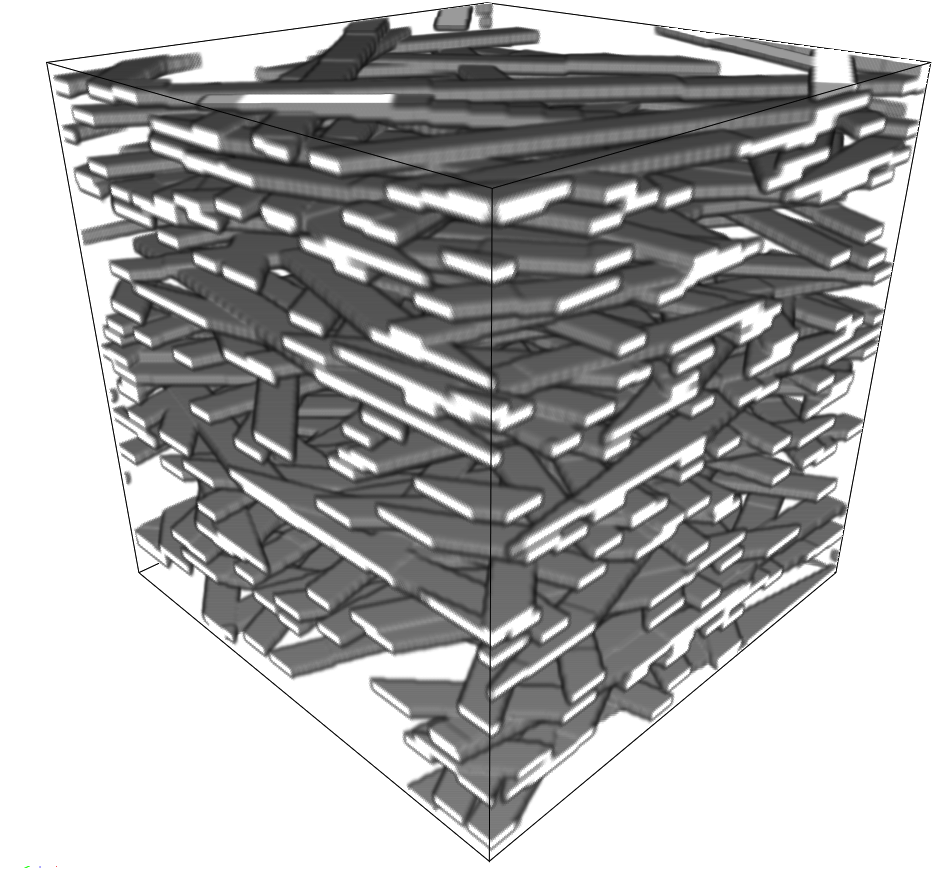}
    \caption{Left panel: Simulation of an isotropic Poisson cylinder process in $\RR^3$ with spherical cylinder base. Right panel: Simulation of an anisotropic Poisson cylinder process in $\RR^3$ with rectangular cylinder base. Both simulations were provided by Claudia Redenbach, Kaiserslautern.}
    \label{fig}
\end{figure}

The aim of the present paper is to derive tail bounds for the volume as well as for the intrinsic volumes of the random union set $Z$ associated with a stationary Poisson cylinder process restricted to a compact observation window. More precisely, under rather general assumptions on the distribution of the typical cylinder base we derive bounds for the upper and lower tail of the volume ($d$-dimensional Lebesgue measure) $F:=\lambda_d(Z\cap W)$ of $Z\cap W$, where $W\subset\RR^d$ is a compact set with positive volume. Our bounds generalize in a natural way the results from \cite{GieringerLast} for the Boolean model. A number of special cases are discussed separately. For example, we consider the case where the cylinder bases are random rotations of a fixed convex body. We will see that in this situation our tail bounds are of the form 
\begin{align}
    \PP(F-\EE F\geq r) &\leq \exp(-\boldsymbol{\Theta}(r\log r)),\qquad r\geq 0,\label{eq:IntroUT}\\
    \PP(F-\EE F\leq -r) &\leq \exp(-\boldsymbol{\Theta}(r^2)),\qquad 0\leq r\leq \EE F,\label{eq:IntroLT}
\end{align}
where $\boldsymbol{\Theta}(r\log r)$ stands for a quantity from $\boldsymbol{O}(r\log r)\cap\boldsymbol{\Omega}(r\log r)$ in the usual Landau notation. As for the Boolean model this constitutes a significant improvement compared to the bounds that can be deduced by means of the general limit theorems for large deviations \cite{Saulis} from the cumulant estimates provided in \cite{HeinrichSpiessCLTVolume}.

Beside the volume of $Z\cap W$ we also study the intrinsic volumes of $Z\cap W$ under the assumption that the cylinder bases are convex and that the union set $Z$ is a stationary and isotropic random closed set. We emphasize that the intrinsic volumes are of particular importance since every continuous, additive and motion-invariant functional on the class of convex bodies can be represented as a linear combination of intrinsic volumes (this is the content of Hadwiger's theorem). We remark that compared to the volume case the intrinsic volumes are more difficult to handle. This partially relies on the fact that even mean value formulas for intrinsic volumes of (stationary and isotropic) Poisson cylinder processes are not available in the existing literature and needed to be developed in the present paper as well. In addition, for the case of intrinsic volumes, isoperimetric inequalities have to be used in order to bring the bounds in a convenient form. As for the volume we consider especially the case where the cylinder bases are random rotations of a fixed convex body and deduce bounds which are comparable to \eqref{eq:IntroUT} and \eqref{eq:IntroLT}. We remark that our results for intrinsic volumes are new even for the special case of the Boolean model for which only the case of the surface content was previously studied in \cite{GieringerDissertation} under quite restrictive assumptions on the typical cylinder base.

\medspace

The remaining parts of this paper are structured as follows. In Section \ref{sec:Notation} we gather some notation and in Section \ref{sec:PCP} we recall the formal definition and description of a Poisson cylinder process and its associated union set. In particular, we derive there a necessary and sufficient criterion under which the union set is isotropic. A concentration inequality for general Poisson functionals from \cite{GieringerLast} is presented in Section \ref{sec:GeneralConcentration}. Tail bounds for the volume are the content of Section \ref{SecVolume} and a number of special cases are discussed in Section \ref{sec:SpecialCasesVolume}. We present concentration properties for the class of intrinsic volumes in the final Section \ref{sec:IntVol}.

\section{Preliminaries}

\subsection{General notation}\label{sec:Notation}

For $d\in\NN$ we let $\lambda_d$ be the Lebesgue measure on $\RR^d$. The $s$-dimensional Hausdorff measure is denoted by $\cH^s$, $s\geq 0$. A centred Euclidean ball in $\RR^d$ with radius $r>0$ is denoted by $B_r^d$. The volume of the $d$-dimensional unit ball is given by $\kappa_d:=\lambda_d(B_1^d)=\frac{\pi^{d/2}}{\Gamma(1+d/2)}$. We let $\cC'(\RR^d)$ be the space of non-empty compact subsets of $\RR^d$ and recall that by a convex body $K\subset\RR^d$ we understand a compact convex set with non-empty interior. For $k\in\{0,1,\ldots,d\}$ and a convex body $K\subset\RR^d$ we let $V_j(K)$ be the $j$th intrinsic volume of $K$. In particular, $V_d(K)=\lambda_d(K)$, $V_{d-1}(K)={1\over 2}\cH^{d-1}(\partial K)$ and $V_1(K)$ is a constant multiple of the mean width of $K$. We use the symbol $\diam(A)$ to indicate the diameter of a set $A\subset\RR^d$. For a (possibly lower-dimensional) convex set $K\subset\RR^d$ we denote by $K^*=-K$ the reflection of $K$ at the origin. Moreover, the linear hull of $A\subset\RR^d$ is denoted by ${\rm lin}(A)$. By $P_{d-k}:\RR^d\to\RR^{d-k}$ we denote the orthogonal projection of $\RR^d$ to $\RR^{d-k}$, i.e., the projection to the first $d-k$ coordinates. By $\OO_d$ and $\SS\OO_d$ we denote the group of orthogonal $d\times d$ matrices and of orthogonal $d\times d$ matrices with determinant $1$, respectively. 

\subsection{Poisson cylinder processes}\label{sec:PCP}

Let $d\geq 2$ and $k\in\{0,1,\ldots,d-1\}$.  Further, we let $G(d,k)$ be the Grassmannian of $k$-dimensional linear subspaces of $\RR^d$. By a $k$-cylinder in $\RR^d$ one understands the Minkowski sum of some $L\in G(d,k)$ with a non-empty compact subset of $L^\perp$, the orthogonal complement of $L$. We identify a subspace $L\in G(d,k)$ with the unique element $\phi_L$ of the equivalence class $\Phi_L$ of orthogonal matrices $\phi\in \SS\OO_d$ satisfying $L=\phi E_k$, where $E_k={\rm lin}(e_{d-k+1},\ldots,e_d)$ and $e_1,\ldots,e_d$ is the standard orthonormal basis in $\RR^d$. In fact, one can choose for $\phi_L$ the lexicografically smallest element of the compact set $\Phi_L$, which yields a one-to-one correspondence between $G(d,k)$ and $\SS\OO_{d,k}:=\{\phi_L={\rm lex\,min}\,\Phi_L:L\in G(d,k)\}$ up to orientation of the subspaces, cf. \cite{HeinrichSpiessCLTVolume,HeinrichSpiessCLTVundS}. In particular, this allows us to regard $\SS\OO_{d,k}$ as a compact homogeneous space for $\SS\OO_d$.

Fix $\gamma\in(0,\infty)$ and let $\eta$ be a stationary Poisson process on ${\rm lin}(e_1,\ldots,e_{d-k})\subset\RR^d$ with intensity $\gamma$. Let $\cC_{d-k}'$ be the space of non-empty compact subsets of ${\rm lin}(e_1,\ldots,e_{d-k})$. Here and in what follows, we identify ${\rm lin}(e_1,\ldots,e_{d-k})$ with $\RR^{d-k}\subset\RR^d$. We define $\MM_{d,k}:=\SS\OO_{d,k}\times\cC_{d-k}'$ and let $\QQ$ be a probability measure on $\MM_{d,k}$. By $\xi$ we denote an independent $\QQ$-marking of $\eta$, which is a Poisson process on the product space $\RR^{d-k}\times\MM_{d,k}$ with intensity measure $\gamma\,\lambda_{d-k}\otimes\QQ$, cf. \cite{LP}. Further we denote by $(\Theta,\Xi)\in\MM_{d,k}$ a random pair with distribution $\QQ$. It represents the (not necessarily independent) distribution of the direction and the base of the typical cylinder in the usual sense of Palm theory. By a stationary Poisson process of $k$-cylinders with intensity $\gamma$ and base-direction distribution $\QQ$ we understand the point process
$$
\widetilde{\xi}:=\sum_{(x,\theta,K)\in\xi}\delta_{Z(x,\theta,K)},\qquad Z(x,\theta,K)=\theta((K+x)\times E_k)
$$
on the space of $k$-cylinders in $\RR^d$, where $\delta_{(\,\cdot\,)}$ denotes the Dirac measure, cf.\ \cite{HeinrichSpiessCLTVolume,HeinrichSpiessCLTVundS}. In this paper we are interested in the random union set
$$
Z := \bigcup_{X\in\widetilde{\xi}}X=\bigcup_{(x,\theta,K)\in\xi}Z(x,\theta,K)
$$
induced by the stationary marked Poisson process $\xi$ or the stationary Poisson cylinder process $\widetilde{\xi}$, respectively, where we write $X\in \widetilde{\xi}$ to indicate that $X$ belongs to the support of $\widetilde{\xi}$. It is known from \cite{HeinrichSpiessCLTVolume,HeinrichSpiessCLTVundS} that $Z$ is a random closed subset of $\RR^d$ in the usual sense of stochastic geometry \cite[Chapter 2]{SW}, provided that 
\begin{align}\label{eq:ConditionRACS}
\EE\lambda_{d-k}(\Xi+B_\varepsilon^{d-k})<\infty\qquad\text{for some}\ \varepsilon>0.
\end{align}
In this case, $F:=\lambda_d(Z\cap W)$ is a well-defined random variable for any compact subset $W\subset\RR^d$. In what follows we shall assume that \eqref{eq:ConditionRACS} is always satisfied.

Another point we shall discuss here is the isotropy property of the random union set $Z$, which means that $\rho Z$ has the same distribution as $Z$ for all $\rho\in \SS\OO_d$. While for the Boolean model an isotropy criterium is well known, surprisingly we were not able to locate a necessary and sufficient condition for isotropy of $Z$ in the existing literature.

\begin{lemma}\label{lem:isotropy}
The random closed set $Z$ is isotropic if and only if $\QQ(\SS\OO_{d,k}\times\,\cdot\,)$ is an $\OO_{d-k}$-invariant probability measure on $\mathcal{C}_{d-k}'$ and $\QQ(\,\cdot\,\times\mathcal{C}_{d-k}')$ is the $\SS\OO_d$-invariant Haar probability measure on $\SS\OO_{d,k}$.
\end{lemma}
\begin{proof}
We recall that the capacity functional $T_X$ of a random closed set $X$ is given by $T_X(C):=\PP(X\cap C\neq\varnothing)$, $C\in\mathcal{C}'(\RR^d)$. According to \cite[Theorem 2.4.5]{SW} a random closed set $X$ is isotropic if and only if its capacity functional is rotation invariant, that is, if $T_X(C)=T_X(\rho C)$ holds for all $\rho\in \SS\OO_d$ and $C\in\mathcal{C}'(\RR^d)$. The capacity functional $T_Z$ of $Z$ is known and given by 
\[
T_Z(C)= 1-\exp\big(-\gamma\EE\lambda_{d-k}(P_{d-k}(\Theta^TC)+\Xi^*)\big)
\]
according to \cite[Lemma 1]{SpiessSpodarev} or the results in \cite[Section 5]{HeinrichSpiessCLTVolume}.
Now, for $\rho \in \SS\OO_d$ consider 
\[
1- T_Z(\rho C)= \exp\big(-\gamma\EE\lambda_{d-k}(P_{d-k}(\Theta^T(\rho C))+\Xi^*)\big)
\]
and note that
\begin{align}
\nonumber \EE\lambda_{d-k}\left(P_{d-k}(\Theta^T(\rho C))+\Xi^*\right)&=\int\limits_{\MM_{d,k}}\lambda_{d-k}\left(P_{d-k}(\theta^T(\rho C))+K^*\right)\QQ(\dint(\theta,K))\\
&=\int\limits_{\MM_{d,k}}\lambda_{d-k}\left(P_{d-k}((\rho^T\theta)^T(C))+K^*\right)\QQ(\dint(\theta,K)).\label{eq:26-06-19-5}
\end{align}
It was mentioned in \cite{HeinrichSpiessCLTVolume} that the space $\SS\OO_{d,k}$ is the same as the space of representatives of the quotient space $\SS\OO_d/\SS(\OO_{d-k}\times \OO_k)$, where $\SS(\OO_{d-k}\times \OO_k)$ can be identified with the following space of block matrices:
\[
\SS(\OO_{d-k}\times \OO_k)=\left\{
\left(\begin{matrix}
A & 0\\
0 & B
\end{matrix} \right)
\colon A \in \OO_{d-k},\; B\in \OO_k,\; \det A = \det B\right\}.
\]
By construction of $\SS\OO_{d,k}$ as a space of canonical representatives, this means that every element $\rho \in \SS\OO_d$ admits a unique decomposition
\begin{align}\label{eq:26-06-19-3}
\rho^T = \rho_{d,k}\rho_{d-k}\rho_{k}    
\end{align}
where $\rho_{d,k}\in \SS\OO_{d,k}$, $\rho_{d-k}\in \widetilde \OO_{d-k}$ and $\rho_{k}\in \widetilde \OO_{k}$. Here, $\widetilde{\OO}_{d-k}$ and $\widetilde{\OO}_{k}$ are the sets of block matrices given by
\[
\widetilde{\OO}_{d-k}:=\left\{
\left(\begin{matrix}
A & 0\\
0 & I_k
\end{matrix} \right)
\colon A \in \OO_{d-k}\right\},
\qquad
\widetilde{\OO}_{k}:=\left\{
\left(\begin{matrix}
I_{d-k} & 0\\
0 & B
\end{matrix} \right)
\colon B \in \OO_{k}\right\}
\]
with $I_n$ being $n\times n$ identity matrix, $n\in\NN$. For any $\theta\in \SS\OO_{d,k}$ there are uniquely determined elements $\rho_{d,k}^\theta\in \SS\OO_{d,k}$, $\rho_{d-k}^\theta\in\widetilde{\OO}_{d-k}$ and $\rho_k^\theta\in\widetilde{\OO}_k$ such that
\begin{align}\label{eq:26-06-19-4}
\rho^T\theta = \rho_{d,k}^\theta\rho_{d-k}^\theta\rho_k^\theta.
\end{align}
Plugging this into \eqref{eq:26-06-19-5} yields
\begin{align*}
    \EE\lambda_{d-k}\left(P_{d-k}(\Theta^T(\rho C))+\Xi^*\right)&=\int\limits_{\MM_{d,k}}\lambda_{d-k}\left(P_{d-k}((\rho_{d,k}^\theta\rho_{d-k}^\theta\rho_k^\theta)^T(C))+K^*\right)\QQ(\dint(\theta,K))\\
    &=\int\limits_{\MM_{d,k}}\lambda_{d-k}\left(P_{d-k}((\rho_k^\theta)^T(\rho_{d-k}^\theta)^T(\rho_{d,k}^\theta)^T(C))+K^*\right)\QQ(\dint(\theta,K))\\
    &=\int\limits_{\MM_{d,k}}\lambda_{d-k}\left(P_{d-k}((\rho_{d,k}^\theta)^T(C))+(\rho_{d-k}^\theta)(K^*)\right)\QQ(\dint(\theta,K))\\
    &=\int\limits_{\MM_{d,k}}\lambda_{d-k}\left(P_{d-k}((\rho_{d,k}^\theta)^T(C))+K^*\right)\QQ(\dint(\theta,K)).
\end{align*}
Here, to obtain the third equality we used the fact that $(\rho_k^\theta)^T$ does not influence the projection $P_{d-k}(\,\cdot\,)$, that $\rho_{d-k}^\theta$ acts in $\RR^{d-k}$ and thus commutes with the projection $P_{d-k}$, and that the Lebesgue measure $\lambda_{d-k}$ is $\OO_{d-k}$-invariant. Moreover, the last equality follows from the fact that $\QQ(\SS\OO_{d,k}\times\,\cdot\,)$ is $\OO_{d-k}$-invariant by assumption.

To simplify the last expression further, we apply twice the change-of-basis formula from linear algebra. This implies the relations
$$
\rho_{d-k}^\theta = \theta^T\rho_{d-k}\theta\qquad\text{and}\qquad\rho_{k}^\theta = \theta^T\rho_{k}\theta.
$$
Putting this together with \eqref{eq:26-06-19-3} and \eqref{eq:26-06-19-4} we conclude that
\begin{align*}
    \rho_{d,k}\rho_{d-k}\rho_k\theta = \rho^T\theta = \rho_{d,k}^\theta(\theta^T\rho_{d-k}\theta)(\theta^T\rho_k\theta) = \rho_{d,k}^\theta\theta^T\rho_{d-k}\rho_k\theta,
\end{align*}
and hence $\rho_{d,k}^\theta=\rho_{d,k}\theta$. From this we obtain
\begin{align*}
    1-T_Z(\rho C) &= \exp\Big(-\gamma\int\limits_{\MM_{d,k}}\lambda_{d-k}\left(P_{d-k}((\rho_{d,k}\theta)^T(C))+K^*\right)\QQ(\dint(\theta,K))\Big)\\
    &= \exp\Big(-\gamma\int\limits_{\MM_{d,k}}\lambda_{d-k}\left(P_{d-k}(\theta^T(C))+K^*\right)\QQ(\dint(\theta,K))\Big)\\
    &=1-T_Z(C),
\end{align*}
where we have used our assumption that $\QQ(\,\cdot\,\times\mathcal{C}_{d-k}')$ is the $\SS\OO_d$-invariant Haar probability measure on $\SS\OO_{d,k}$. This concludes the proof.
\end{proof}

For a measurable set $M\subseteq\MM_{d,k}$ and $\rho\in \SS\OO_d$ we define
$$
\rho M := \{(\rho^\theta_{d,k},\rho^{\theta}_{d-k}K):(\theta,K)\in M\},
$$
where we applied the decomposition $\rho\theta=\rho^\theta_{d,k}\rho^\theta_{d-k}\rho^\theta_k$ with $\rho^\theta_{d,k}\in \SS\OO_{d,k}$, $\rho^\theta_{d-k}\in\widetilde{\OO}_{d-k}$ and $\rho^\theta_k\in\widetilde{\OO}_k$ using the notation introduced in the previous proof. We say that a probability measure $\QQ$ on $\MM_{d,k}$ is rotation invariant, provided that $\QQ(\rho M)=\QQ(M)$ for all measurable $M\subseteq\MM_{d,k}$. Repeating the same argument as in the proof of Lemma \ref{lem:isotropy} we can conclude that $\QQ$ is rotation invariant if and only if $\QQ(\SS\OO_{d,k}\times\,\cdot\,)$ is an $\OO_{d-k}$-invariant probability measure on $\mathcal{C}_{d-k}'$ and $\QQ(\,\cdot\,\times\mathcal{C}_{d-k}')$
is the $\SS\OO_d$-invariant Haar probability measure on $\SS\OO_{d,k}$. Especially, from Lemma \ref{lem:isotropy} we conclude that the random union set $Z$ is isotropic if and only if $\QQ$ is rotation invariant.

\subsection{Concentration inequalities for general Poisson functionals}\label{sec:GeneralConcentration}

In this section we rephrase a general concentration inequality for Poisson functionals that was recently proved in \cite{GieringerLast} using a covariance identity for functionals of Poisson processes on abstract phase spaces and a classical Chernoff-type argument. For this we let $(\XX,\cX)$ be a measurable space and $\Lambda$ be some $\sigma$-finite measure on $\XX$. By $\eta$ we denote a Poisson process on $\XX$ with intensity measure $\Lambda$, which is defined over some probability space $(\Omega,\cA,\PP)$, cf.\ \cite{LP}. By $\bN=\bN(\cX)$ we denote the space of $\sigma$-finite counting measures on $\XX$, which is supplied with the $\sigma$-field $\cN$ induced by the vague topology on $\bN$. We denote the distribution on $\bN$ of a Poisson process with intensity measure $\Lambda$ by $\Pi_\Lambda$. Finally, by a Poisson functional we understand a random variable $F$ $\PP$-almost surely satisfying $F=f(\eta)$ for some measurable function $f:\bN\to\RR$, called a representative of $F$.

For a measurable function $f:\bN\to\RR$ and a point $x\in\XX$ we define the first-order difference (or add-one-cost) operator by
$$
D_xf(\mu) = f(\mu+\delta_x)-f(\mu),\qquad\mu\in\bN.
$$
In particular, for a Poisson functional $F$ with representative $f$ we write $D_xF$ for $D_xf(\eta)$. For a square-integrable Poisson functional $F\in L^2(\PP)$ we define
\begin{align}\label{eq:DefSF}
s_F := \sup\{s\geq 0:e^{sF}\in L^2(\PP),De^{sF}\in L^2(\PP\otimes\Lambda)\}\in[0,\infty],
\end{align}
and for $s\in[0,s_F)$ we put
\begin{align}\label{eq:DefVF}
V_F(s) := \int\limits_{\XX}(e^{sD_xF}-1)\int\limits_0^1\int\limits_{\bN}D_xf(\eta_t+\mu)\,\Pi_{(1-t)\Lambda}(\dint\mu)\,\dint t\,\Lambda(\dint x),
\end{align}
where $\eta_t$, $t\in[0,1]$ denotes a $t$-thinning of $\eta$ (which is a Poisson process on $\XX$ with intensity measure $t\Lambda$). We are now in the position to rephrase the concentration inequality from \cite[Corollary 2.3]{GieringerLast}.

\begin{lemma}\label{lem:GieringerLast}
Let $F=f(\eta)\in L^2(\PP)$ be a Poisson functional such that $DF\in L^2(\PP\otimes\Lambda)$. Assume that $\PP$-almost surely $V_F(s)\leq v(s)$ for some measurable function $v:[0,s_F)\to\RR$. Then
$$
\PP(F-\EE F\geq r) \leq \exp\Big(\inf_{s\in[0,s_F)}\Big(\int\limits_0^sv(u)\,\dint u-rs\Big)\Big),\qquad\qquad r\geq 0.
$$
\end{lemma}

As already remarked in \cite{GieringerLast} a similar inequality holds for the lower tail of the distribution of $F$ if $s_F$ and $V_F(s)$ from \eqref{eq:DefSF} and \eqref{eq:DefVF} are replaced by $s_F^{({\rm lt})}:=s_{-F}$ and $V_F^{({\rm lt})}(s):=V_{-F}(s)$, respectively. In particular, note that for $s\in[0,s_F^{({\rm lt})})$ the identity
$$
V_F^{({\rm lt})}(s) = \int\limits_{\XX}(1-e^{-sD_xF})\int\limits_0^1\int\limits_{\bN}D_xf(\eta_t+\mu)\,\Pi_{(1-t)\Lambda}(\dint\mu)\,\dint t\,\Lambda(\dint x)
$$
holds.

\begin{lemma}\label{lem:GieringerLastLOWERTAIL}
Let $F=f(\eta)\in L^2(\PP)$ be such that $DF\in L^2(\PP\otimes\Lambda)$. Assume that $\PP$-almost surely $V_F^{({\rm lt})}(s)\leq v(s)$ for some measurable function $v:[0,s_F^{({\rm lt})})\to\RR$. Then
$$
\PP(F-\EE F\leq -r) \leq \exp\Big(\inf_{s\in[0,s_F^{({\rm lt})})}\Big(\int\limits_0^sv(u)\,\dint u-rs\Big)\Big),\qquad\qquad r\geq 0.
$$
\end{lemma}

\section{A concentration inequality for the volume}\label{SecVolume}

Our goal in this section is to apply the concentration inequalities for general Poisson functionals from Section \ref{sec:GeneralConcentration} to the volume of the union set of a stationary Poisson cylinder process within a bounded window. More precisely, we let $Z$ be the union set of a stationary Poisson process of $k$-cylinders in $\RR^d$ with intensity $\gamma\in(0,\infty)$ and base-direction distribution $\QQ$. We assume that all random quantities considered are defined over some probability space $(\Omega,\cA,\PP)$. Moreover, we let $W\subset\RR^d$ be a compact set with $\lambda_d(W)>0$. We are interested in the Poisson functional
$$
F := \lambda_d(Z\cap W),
$$
i.e., $F$ is the total volume ($d$-dimensional Lebesgue measure) of all cylinders within $W$. In this section and the next section we assume that the volume of the typical cylinder base has positive and finite first moment, i.e., 
$$
m_{d-k}:=\EE\lambda_{d-k}(\Xi)\in (0,\infty).
$$

In order to check the assumptions of Lemma \ref{lem:GieringerLast}, an analysis of the first-order difference operator of $F$ is necessary. We start by observing that the additivity of the Lebesgue measure implies that, for $(\lambda_{d-k}\otimes\QQ)$-almost all $(x,\theta,K)\in\RR^{d-k}\times\MM_{d,k}$, the following equality
\begin{align*}
D_{(x,\theta,K)}F &= \lambda_d((Z\cup Z(x,\theta,K))\cap W) - \lambda_d(Z\cap W) \\
&=\lambda_d(Z\cap W)+\lambda_d(Z(x,\theta,K)\cap W)-\lambda_d(Z\cap Z(x,\theta,K)\cap W)-\lambda_d(Z\cap W)\\
&= \lambda_d(Z(x,\theta,K)\cap W) - \lambda_d(Z\cap Z(x,\theta,K)\cap W)
\end{align*}
holds $\PP$-almost surely. Using this representation for the difference operator, we can prove the following technical result, where we recall the definition of the quantity $s_F$ from \eqref{eq:DefSF} and also that $s_F^{({\rm lt})}=s_{-F}$.

\begin{lemma}\label{lem:AssumptionsVolume}
Under the assumptions mentioned above we have that $F\in L^2(\PP)$, $DF\in L^2(\PP\otimes\lambda_{d-k}\otimes\QQ)$, $s_F=\infty$ and $s_F^{({\rm lt})}=\infty$.
\end{lemma}
\begin{proof}
Corollary 18.8 in \cite{LP} shows that
\begin{align*}
\EE(F^2) &\leq (\EE(F))^2 +  \gamma\int\limits_{\RR^{d-k}}\int\limits_{\MM_{d,k}}\EE[(D_{(x,\theta,K)}F)^2]\,\QQ(\dint(\theta,K))\,\lambda_{d-k}(\dint x)\\
&\leq \lambda_{d}(W)^2 + \gamma\int\limits_{\RR^{d-k}}\int\limits_{\MM_{d,k}}\lambda_d(Z(x,\theta,K)\cap W)^2\,\QQ(\dint(\theta,K))\,\lambda_{d-k}(\dint x)\\
&\leq \lambda_{d}(W)^2+\gamma\,\lambda_d(W)^2\,(\lambda_{d-k}\otimes\QQ)\left(\{(x,\theta,K)\in\RR^{d-k}\times\MM_{d,k}:Z(x,\theta,K)\cap W\neq\varnothing\}\right)<\infty,
\end{align*}
since $W$ is compact and $\lambda_{d-k}\otimes\QQ$ is a locally finite measure on $\RR^{d-k}\times\MM_{d,k}$. This shows that $F\in L^2(\PP)$ and at the same time $DF\in L^2(\PP\otimes\lambda_{d-k}\otimes\QQ)$. 

Next, we let $s\geq 0$ and observe that $\PP$-almost surely and for $(\lambda_{d-k}\otimes\QQ)$-almost all $(x,\theta,K)\in\RR^{d-k}\times\MM_{d,k}$,
\begin{align*}
D_{(x,\theta,K)}e^{sF} &= e^{s\lambda_d((Z\cup Z(x,\theta,K))\cap W)}-e^{s\lambda_d(Z\cap W)}=e^{sF}\left(e^{sD_{(x,\theta,K)}F}-1\right).
\end{align*}
Thus,
\begin{align*}
&\EE\int\limits_{\RR^{d-k}}\int\limits_{\MM_{d,k}}\big(D_{(x,\theta,K)}e^{sF}\big)^2\,\QQ(\dint(\theta,K))\,\lambda_{d-k}(\dint x)\\
&=\int\limits_{\RR^{d-k}}\int\limits_{\MM_{d,k}}\EE\big[\big(e^{sF}(e^{sD_{(x,\theta,K)}F}-1)\big)^2\big]\,\QQ(\dint(\theta,K))\,\lambda_{d-k}(\dint x)\\
&\leq \EE[e^{2sF}]\int\limits_{\RR^{d-k}}\int\limits_{\MM_{d,k}}\big(e^{s\lambda_d(Z(x,\theta,K)\cap W)}-1\big)^2\,\QQ(\dint(\theta,K))\,\lambda_{d-k}(\dint x)\\
&\leq e^{2s\lambda_d(W)}\big(e^{s\lambda_d(W)}-1\big)^2\,(\lambda_{d-k}\otimes\QQ)\left(\{(x,\theta,K)\in\RR^{d-k}\times\MM_{d,k}:Z(x,\theta,K)\cap W\neq\varnothing\}\right).
\end{align*}
Since the last expression is finite for all $s\geq 0$, we conclude that $De^{sF}\in L^2(\PP\otimes\lambda_{d-k}\otimes\QQ)$. Moreover, using that
\begin{align*}
&\EE\int\limits_{\RR^{d-k}}\int\limits_{\MM_{d,k}}e^{2sF}\,\QQ(\dint(\theta,K))\,\lambda_{d-k}(\dint x)\\
&\qquad\leq e^{2s\lambda_d(W)}\,(\lambda_{d-k}\otimes\QQ)\left(\{(x,\theta,K)\in\RR^{d-k}\times\MM_{d,k}:Z(x,\theta,K)\cap W\neq\varnothing\}\right),
\end{align*}
we obtain the assertion that $s_F=\infty$. To prove that $s_F^{(\rm lt)}=\infty$, we first observe that, since $\PP$-almost surely $e^{-sF}\leq 1$, necessarily $e^{-sF}\in L^2(\PP)$ for any $s\geq 0$. In addition, similarly to the above argument, we have that
\begin{align*}
&\EE\int\limits_{\RR^{d-k}}\int\limits_{\MM_{d,k}}\big(D_{(x,\theta,K)}e^{-sF}\big)^2\,\QQ(\dint(\theta,K))\,\lambda_{d-k}(\dint x) \\
&\qquad\leq (\lambda_{d-k}\otimes\QQ)\left(\{(x,\theta,K)\in\RR^{d-k}\times\MM_{d,k}:Z(x,\theta,K)\cap W\neq\varnothing\}\right)<\infty
\end{align*}
for any $s\geq 0$. This shows that $s_F^{(\rm lt)}=\infty$.
\end{proof}

Recall that 
\begin{equation}\label{eq:ExpectationF}
\EE F=\lambda_d(W)(1-e^{-\gamma\,m_{d-k}})=:\lambda_d(W)\,p,    
\end{equation}
where $p=1-e^{-\gamma\,m_{d-k}}$ is the volume fraction of the random union set $Z$, see \cite{HeinrichSpiessCLTVolume,SpiessSpodarev}.
The main result of this section is the following bound for the upper and the lower tail of the volume of the union set of a stationary Poisson cylinder process in a window $W$. 

\begin{theorem}\label{thm:GeneralConcentrationInequ}
For all $r\geq 0$, one has that
\begin{align*}
\PP(F-\EE F\geq r) &\leq \exp\Big(\inf_{s\geq 0}\Big({p\over m_{d-k}}\EE\Big[\lambda_{d-k}(P_{d-k}(\Theta^TW)+\Xi^*)\Psi(s\lambda_{d-k}(\Xi)\diam(W)^k)\Big]-rs\Big)\Big),
\end{align*}
and for $0\leq r\leq\EE F$ one has that
\begin{align*}
\PP(F-\EE F\leq -r) &\leq \exp\Big(\inf_{s\geq 0}\Big({p\over m_{d-k}}\EE\Big[\lambda_{d-k}(P_{d-k}(\Theta^TW)+\Xi^*)\Psi(-s\lambda_{d-k}(\Xi)\diam(W)^k)\Big]-rs\Big)\Big),
\end{align*}
where $\Psi(x)=e^x-x-1$, $x\in\RR$.
\end{theorem}

\begin{remark}\rm 
\begin{itemize}
    \item[(i)] Note that the condition $0\leq r\leq\EE F$ in the inequality for the lower tail is not strictly necessary. However, the inequality becomes trivial for all $r>\EE F$, since $F$ can only take non-negative values.
    
    \item[(ii)] Taking $k=0$, $Z$ is nothing else than a Boolean model based on a stationary Poisson point process on $\RR^d$ with intensity $\gamma$ and typical grain $\Xi$. In this case the two inequalities in Theorem \ref{thm:GeneralConcentrationInequ} reduce to
\begin{align}\label{eq:ConcBMVolume}
\PP(F-\EE F\geq r) &\leq \exp\Big(\inf_{s\geq 0}\Big({p\over m_{d}}\EE[\lambda_{d}(W+\Xi^*)\Psi(s\lambda_{d}(\Xi))]-rs\Big)\Big),\qquad r\geq 0,
\end{align}
and
\begin{align*}
\PP(F-\EE F\leq -r) &\leq \exp\Big(\inf_{s\geq 0}\Big({p\over m_{d}}\EE[\lambda_{d}(W+\Xi^*)\Psi(-s\lambda_{d}(\Xi))]-rs\Big)\Big),\qquad 0\leq r\leq\EE F.
\end{align*}
In this form they are known from \cite{GieringerDissertation,GieringerLast} and our result can be seen as a natural generalization to general $k\in\{1,\ldots,d-1\}$.

\item[(iii)] Consider the degenerate case where $\PP$-almost surely $\Xi=\{0\}$, which is not covered by Theorem \ref{thm:GeneralConcentrationInequ}. Then $Z$ is the union set associated with a stationary Poisson process of $k$-flats in $\RR^d$, see \cite[Section 4.4]{SW}. For simplicity assume that $Z$ is isotropic. In this case one can consider for compact and convex $W\subset\RR^d$ with $\lambda_d(W)>0$ the $k$-dimensional Hausdorff measure $F=\cH^k(Z\cap W)$. The difference operator is then given by $\cH^k(W\cap\theta(x+E_k))$, independently of $Z$. One can thus apply the general concentration inequality \cite[Proposition 3.1]{Wu} in combination with Crofton's formula \cite[Theorem 5.1.1]{SW} from integral geometry to conclude that
$$
\PP(F-\EE F\geq r) \leq \exp\Big(-{r\over 2b}\log\Big(1+{br\over a^2}\Big)\Big),\qquad r\geq 0
$$
with $b=\big(\frac{\diam(W)}{2}\big)^k \kappa_{k}$ and $a=\gamma{\kappa_k^3\kappa_{d-k}\over{d\choose k}\kappa_d}\big(\frac{\diam(W)}{2}\big)^{2k}V_{d-k}(W)$. With different constants $a$ and $b$ this can also be established along the lines of the proof of Theorem \ref{thm:GeneralConcentrationInequ}.
\end{itemize}
\end{remark}

\begin{proof}[Proof of Theorem \ref{thm:GeneralConcentrationInequ} -- Upper tail]
In the firs step, we deduce an upper bound for the function $V_F(s)$, $s\geq 0$, defined in \eqref{eq:DefVF}. We start by putting $\Lambda:=\gamma\,\lambda_{d-k}\otimes\QQ$ and considering the term
\begin{align*}
T_t:=\int\limits_{\bN}D_{(x,\theta,K)}f(\eta_t+\mu)\,\Pi_{(1-t)\Lambda}(\dint\mu),\qquad t\in[0,1],
\end{align*}
where we denote by $f$ a representative of $F$. By the superposition property of Poisson processes, the inequality
$$
D_{(x,\theta,K)}f(\eta_t+\mu) \leq D_{(x,\theta,K)}f(\mu)
$$
holds for $(\lambda_{d-k}\otimes\QQ)$-almost all $(x,\theta,K)$ and all $t\in[0,1]$. Thus, we have that
\begin{align*}
T_t &\leq \lambda_d(Z(x,\theta,K)\cap W)-\int\limits_{\bN}\lambda_d(Z(\mu)\cap Z(x,\theta,K)\cap W)\,\Pi_{(1-t)\Lambda}(\dint\mu)\\
&=\lambda_d(Z(x,\theta,K)\cap W)-\lambda_d(Z(x,\theta,K)\cap W)\big(1-e^{-(1-t)\,\gamma\,\EE\lambda_{d-k}(\Xi)}\big)\\
&=\lambda_d(Z(x,\theta,K)\cap W)\,e^{-(1-t)\gamma\,m_{d-k}}.
\end{align*}
As a consequence and by using Fubini's theorem, we find that
\begin{align*}
V_F(s) &\leq \gamma\int\limits_{\RR^{d-k}}\int\limits_{\MM_{d,k}}\big(e^{s\lambda_d(Z(x,\theta,K)\cap W)}-1\big)\int\limits_0^1 T_t\,\dint t\,\QQ(\dint(\theta,K))\,\lambda_{d-k}(\dint x)\\
&\leq \gamma\int\limits_{\RR^{d-k}}\int\limits_{\MM_{d,k}}\big(e^{s\lambda_d(Z(x,\theta,K)\cap W)}-1\big)\\
&\qquad\qquad\qquad\times\int\limits_0^1\lambda_d(Z(x,\theta,K)\cap W)\,e^{-(1-t)\gamma\,m_{d-k}}\,\dint t\,\QQ(\dint(\theta,K))\,\lambda_{d-k}(\dint x)\\
& = {p\over m_{d-k}}\,v(s),
\end{align*}
where
$$
v(s):=\int\limits_{\RR^{d-k}}\int\limits_{\MM_{d,k}}\big(e^{s\lambda_d(Z(x,\theta,K)\cap W)}-1\big)\,\lambda_d(Z(x,\theta,K)\cap W)\,\QQ(\dint(\theta,K))\,\lambda_{d-k}(\dint x)
$$
and we recall that $p=1-e^{-\gamma\,m_{d-k}}$ is the volume fraction of the random set $Z$.

In a next step, we shall provide an upper bound for the integral
$$
w(s) := {p\over m_{d-k}}\int\limits_0^s v(u)\,\dint u.
$$
We have that
\begin{align*}
{m_{d-k}\over p}\,w(s) &= \int\limits_{\RR^{d-k}}\int\limits_{\MM_{d,k}}\int\limits_0^s\big(e^{u\lambda_d(Z(x,\theta,K)\cap W)}-1\big)\,\lambda_d(Z(x,\theta,K)\cap W)\,\dint u\,\QQ(\dint(\theta,K))\,\lambda_{d-k}(\dint x)\\
&= \int\limits_{\RR^{d-k}}\int\limits_{\MM_{d,k}}\Big[e^{s\lambda_d(Z(x,\theta,K)\cap W)}-s\lambda_d(Z(x,\theta,K)\cap W)-1\Big]\,\QQ(\dint(\theta,K))\,\lambda_{d-k}(\dint x).
\end{align*}
Since
\begin{equation}\label{eq:EstimateDiamter}
\lambda_d(Z(x,\theta,K)\cap W) \leq \lambda_{d-k}(K)\,\diam(W)^k    
\end{equation}
we obtain, using Fubini's theorem and the fact that the function $\Psi(x)=e^x-x-1$ is increasing for $x>0$, that
\begin{align*}
{m_{d-k}\over p}\,w(s) &\leq \int\limits_{\MM_{d,k}}\Big[e^{s\lambda_{d-k}(K)\diam(W)^k}-s\lambda_{d-k}(K)\diam(W)^k-1\Big]\\
&\qquad\times\int\limits_{\RR^{d-k}}{\bf 1}\{Z(x,\theta,K)\cap W\neq\varnothing\}\,\lambda_{d-k}(\dint x)\,\QQ(\dint(\theta,K)).
\end{align*}
Now, for any fixed $\theta\in \SS\OO_{d,k}$ we have that
\begin{equation}\label{eq:IntegralIdentityCylidners}
\begin{split}
&\int\limits_{\RR^{d-k}}{\bf 1}\{Z(x,\theta,K)\cap W\neq\varnothing\}\,\lambda_{d-k}(\dint x)\\
&\qquad= \int\limits_{\RR^{d-k}}{\bf 1}\{(K+x)\cap P_{d-k}(\theta^TW)\neq\varnothing\}\,\lambda_{d-k}(\dint x)\\
&\qquad=\lambda_{d-k}(P_{d-k}(\theta^TW)+K^*).
\end{split}
\end{equation}
Thus, we obtain
\begin{align*}
w(s) &\leq {p\over m_{d-k}}\int\limits_{\MM_{d,k}}\lambda_{d-k}(P_{d-k}(\theta^TW)+K^*)\,\Psi(s\lambda_{d-k}(K)\diam(W)^k)\,\QQ(\dint(\theta,K))\\
&={p\over m_{d-k}}\EE[\lambda_{d-k}(P_{d-k}(\Theta^TW)+\Xi^*)\Psi(s\lambda_{d-k}(\Xi)\diam(W)^k)].
\end{align*}
Combining now the general concentration inequality in Lemma \ref{lem:GieringerLast} with Lemma \ref{lem:AssumptionsVolume} and the above inequality finishes the proof for the upper tail.
\end{proof}

\begin{proof}[Proof of Theorem \ref{thm:GeneralConcentrationInequ} -- Lower tail]
A slight adaption of the proof for the upper tail also shows the following bound for $V_F^{({\rm lt})}(s)$, which will be used to control the lower tail of the Poisson functional $F$. Namely, replacing $F$ by $-F$ we have already shown in the proof of the bound for the upper tail that
$$
V_F^{({\rm lt})}(s) \leq {p\over m_{d-k}}\,v^{({\rm lt})}(s),
$$
where
$$
v^{({\rm lt})}(s) = \int\limits_{\RR^{d-k}}\int\limits_{\MM_{d,k}}\big(1-e^{-s\lambda_d(Z(x,\theta,K)\cap W)}\big)\,\lambda_d(Z(x,\theta,K)\cap W)\,\QQ(\dint(\theta,K))\,\lambda_{d-k}(\dint x).
$$
Now, we compute 
\begin{align*}
{m_{d-k}\over p}w^{({\rm lt})}(s) &:={m_{d-k}\over p}\int\limits_0^s v^{({\rm lt})}(u)\,\dint u\\
&= \int\limits_{\RR^{d-k}}\int\limits_{\MM_{d,k}}\Big[e^{-s\lambda_d(Z(x,\theta,K)\cap W)}+s\lambda_d(Z(x,\theta,K)\cap W)-1\Big]\,\QQ(\dint(\theta,K))\,\lambda_{d-k}(\dint x).
\end{align*}
Using \eqref{eq:EstimateDiamter}, Fubini's theorem, the fact that the function $\Psi(x)$ is decreasing for $x<0$, and \eqref{eq:IntegralIdentityCylidners} we find that
\begin{align*}
{m_{d-k}\over p}w^{({\rm lt})}(s) &\leq \int\limits_{\MM_{d,k}}\Big[e^{-s\lambda_{d-k}(K)\,\diam(W)^k}+s\lambda_{d-k}(K)\,\diam(W)^k-1\Big]\\
&\qquad\times\int\limits_{\RR^{d-k}}{\bf 1}\{Z(x,\theta,K)\cap W\neq\varnothing\}\,\lambda_{d-k}(\dint x)\,\QQ(\dint(\theta,K))\\
&=\int\limits_{\MM_{d,k}}\lambda_{d-k}(P_{d-k}(\theta^TW)+K^*)\,\Psi(-s\lambda_{d-k}(K)\,\diam(W)^k)\,\QQ(\dint(\theta,K))\\
&=\EE[\lambda_{d-k}(P_{d-k}(\Theta^TW)+\Xi^*)\,\Psi(-s\lambda_{d-k}(\Xi)\,\diam(W)^k)].
\end{align*}
Combining this with Lemma \ref{lem:GieringerLastLOWERTAIL} and Lemma \ref{lem:AssumptionsVolume} we finish the proof of the lower tail.
\end{proof}

\section{Special cases}\label{sec:SpecialCasesVolume}

In this section we will consider a number of special choices for the window $W$ as well as for the distribution of the typical cylinder base $\Xi$ when the general estimates in Theorem \ref{thm:GeneralConcentrationInequ} can be made more explicit. To simplify the discussion, we will also assume that the window $W$ and the typical cylinder base $\Xi$ are convex bodies ($\PP$-almost surely).

\subsection{Randomly dilated and rotated cylinder bases}\label{subsec:RotationDilation}

We start with the situation in which the cylinder bases are random rotations and dilatation of a fixed convex body $M\subset\RR^{d-k}$. More precisely, if $U\in \SS\OO_{d-k}$ is a uniform random rotation in $\RR^{d-k}$ (distributed according to the unique rotationally invariant Haar probability measure $\nu_{d-k}$ on $\SS\OO_{d-k}$) and if $R$ is a non-negative random variable with law $\PP_R$ then $\Xi=U(RM)$, where we assume that $U$ and $R$ are independent. In addition, we assume that the direction $\Theta$ of the typical cylinder base is also uniformly distributed on $\SS\OO_{d,k}$ according to the unique $\SS\OO_d$-invariant Haar probability measure $\nu_{d,k}$ on $\SS\OO_{d,k}$, independently of $U$ and $R$. We note that in this situation $m_{d-k}=\lambda_{d-k}(M)\,\EE R^{d-k}$ and $p=1-e^{-\gamma\lambda_{d-k}(M)\EE R^{d-k}}$.

\begin{corollary}\label{cor:DilatedRotatedCylBase}
Let the assumptions just described prevail. Assume that $\EE[R^{d-k}e^{sR^{d-k}}]<\infty$ for some $s>0$, and assume that $\lambda_{d-k}(M)\in(0,\infty)$. Then
\begin{align*}
\PP(F&-\EE F\geq r) \leq \exp\Big(\inf_{s\geq 0}\Big[{p\over m_{d-k}}\sum_{j=0}^{d-k}{\kappa_j\kappa_{d-j}\over{d\choose j}\kappa_d}V_j(W)V_{d-k-j}(M)\EE[R^{d-k-j}\Psi(\alpha s R^{d-k})]-rs\Big]\Big)
\end{align*}
for any $r\geq 0$, where $\alpha:=\lambda_{d-k}(M)\diam(W)^k$. Moreover, assuming that $\EE R^{2(d-k)}<\infty$ we have that
\begin{align*}
\PP(F&-\EE F\leq -r) \leq \exp\Big(\inf_{s\geq 0}\Big[{p\over m_{d-k}}\sum_{j=0}^{d-k}{\kappa_j\kappa_{d-j}\over{d\choose j}\kappa_d}V_j(W)V_{d-k-j}(M)\EE[R^{d-k-j}\Psi(-\alpha s R^{d-k})]-rs\Big]\Big)
\end{align*}
for any $0\leq r\leq \EE[F]$.
\end{corollary}
\begin{proof}
We need to investigate the term
$$
\EE[\lambda_{d-k}(P_{d-k}(\Theta^TW)+\Xi^*)\Psi(s\lambda_{d-k}(\Xi)\diam(W)^k)],
$$
which shows up in the exponents in Theorem \ref{thm:GeneralConcentrationInequ}. Using Fubini's theorem, the assumed independence properties of $\Theta$ and $\Xi$, the scaling property of the Lebesgue measure and the invariance of the Lebesgue measure under rotations we have that 
\begin{align*}
&\EE[\lambda_{d-k}(P_{d-k}(\Theta^TW)+\Xi^*)\Psi(s\lambda_{d-k}(\Xi)\diam(W)^k)]\\
&\qquad=\int\limits_0^\infty\int\limits_{\SS\OO_{d-k}}\int\limits_{\SS\OO_{d,k}}\lambda_{d-k}(P_{d-k}(\theta^TW)+\varrho(rM^*))\,\Psi(s\lambda_{d-k}(\varrho(rM))\diam(W)^k)\\
&\hspace{3cm}\times\nu_{d,k}(\dint\theta)\,\nu_{d-k}(\dint\varrho)\,\PP_R(\dint r)\\
&\qquad=\int\limits_0^\infty\int\limits_{\SS\OO_{d,k}}\int\limits_{\SS\OO_{d-k}}\lambda_{d-k}(P_{d-k}(\theta^TW)+\varrho(rM^*))\,\Psi(sr^{d-k}\lambda_{d-k}(M)\diam(W)^k)\\
&\hspace{3cm}\times\nu_{d-k}(\dint\varrho)\,\nu_{d,k}(\dint\theta)\,\PP_R(\dint r).
\end{align*}
Since $\Psi(sr^{d-k}\lambda_{d-k}(M)\diam(W)^k)$ is independent of $\varrho$, the inner integral can be evaluated by means of the rotational integral formula from \cite[Theorem 6.1.1]{SW}. This yields
\begin{align*}
&\int\limits_{\SS\OO_{d-k}}\lambda_{d-k}(P_{d-k}(\theta^TW)+\varrho(rM^*))\,\nu_{d-k}(\dint\varrho)\\
&\qquad = \sum_{j=0}^{d-k}{\kappa_{d-k-j}\kappa_j\over{d-k\choose j}\kappa_{d-k}}V_j(P_{d-k}(\theta^TW))\,r^{d-k-j}\,V_{d-k-j}(M),
\end{align*}
where we also used the homogeneity of the intrinsic volumes. Thus,
\begin{align*}
&\EE[\lambda_{d-k}(P_{d-k}(\Theta^TW)+\Xi^*)\Psi(s\lambda_{d-k}(\Xi)\diam(W)^k)]\\
&\qquad=\sum_{j=0}^{d-k}{\kappa_{d-k-j}\kappa_j\over{d-k\choose j}\kappa_{d-k}}V_{d-k-j}(M)\int\limits_0^\infty r^{d-k-j}\Psi(sr^{d-k}\lambda_{d-k}(M)\diam(W)^k)\\
&\hspace{3cm}\times\int\limits_{\SS\OO_{d,k}}V_j(P_{d-k}(\theta^TW))\,\nu_{d,k}(\dint\theta)\,\PP_R(\dint r).
\end{align*}
Using now the mean projection formula for intrinsic volumes \cite[Theorem 6.2.2]{SW} we conclude from the definition of $\SS\OO_{d,k}$ and the uniqueness of Haar measures that
\begin{align*}
\int\limits_{\SS\OO_{d,k}}V_j(P_{d-k}(\theta^TW))\,\nu_{d,k}(\dint\theta) = \int\limits_{G(d,d-k)}V_j(P_L(W))\,\nu_{G(d,d-k)}(\dint L) = {{d-j\choose k}\kappa_{d-j}\kappa_{d-k}\over{d\choose k}\kappa_{d-k-j}\kappa_d}\,V_j(W),
\end{align*}
where $P_L(W)$ denotes the orthogonal projection of $W$ onto $L\in G(d,d-k)$ and $\nu_{G(d,d-k)}$ stands for the unique Haar probability measure on the Grassmannian $G(d,d-k)$. As a consequence, we conclude that
\begin{align*}
&\EE[\lambda_{d-k}(P_{d-k}(\Theta^TW)+\Xi^*)\Psi(s\lambda_{d-k}(\Xi)\diam(W)^k)]\\
&\qquad=\sum_{j=0}^{d-k}{\kappa_j\kappa_{d-j}\over{d\choose j}\kappa_d}V_j(W)V_{d-k-j}(M)\int\limits_0^\infty r^{d-k-j}\Psi(sr^{d-k}\lambda_{d-k}(M)\diam(W)^k)\,\PP_R(\dint r)\\
&\qquad=\sum_{j=0}^{d-k}{\kappa_j\kappa_{d-j}\over{d\choose j}\kappa_d}V_j(W)V_{d-k-j}(M)\,\EE[R^{d-k-j}\Psi(\alpha\,s\,R^{d-k})].
\end{align*}
In the same way one shows that
\begin{align*}
&\EE[\lambda_{d-k}(P_{d-k}(\Theta^TW)+\Xi^*)\Psi(-s\lambda_{d-k}(\Xi)\diam(W)^k)]\\
&\qquad=\sum_{j=0}^{d-k}{\kappa_j\kappa_{d-j}\over{d\choose j}\kappa_d}V_j(W)V_{d-k-j}(M)\,\EE[R^{d-k-j}\Psi(-\alpha\,s\,R^{d-k})].
\end{align*}
Together with Theorem \ref{thm:GeneralConcentrationInequ} this yields the result.
\end{proof}

\begin{remark}\rm 
It should be pointed out that there are only few examples of convex bodies for which the intrinsic volumes are available explicitly. For polytopes, they may be expressed in terms of the volumes of lower-dimensional faces together with the external angle at these faces. For example, for the cube one has that
$$
V_j([0,1]^d) = {d\choose j},\qquad\qquad j\in\{0,1,\ldots,d\}.
$$
On the other hand, for the $d$-dimensional unit ball $B^d$ one easily verifies that
$$
V_j(B^d) = {\kappa_d\over\kappa_{d-j}}{d\choose j},\qquad\qquad j\in\{0,1,\ldots,d\}.
$$
\end{remark}
\subsection{Randomly rotated cylinder bases}\label{subsec:Rotation}

In this section we assume that the window $W$ is a general convex body in $\RR^d$, but we strengthen the assumptions on the typical cylinder base by assuming that $\Xi$ arises from a fixed convex body $M\subset\RR^{d-k}$ by a uniform random rotation in $\RR^{d-k}$, that is, we assume that $\Xi=UM$, where $U\in SO_{d-k}$ is distributed according to the Haar measure $\nu_{d-k}$. Note that in this case $m_{d-k}=\lambda_{d-k}(M)$ and $p=1-e^{-\gamma\lambda_{d-k}(M)}$.

\begin{corollary}\label{cor:RotatedCylBase}
Let the assumptions just described prevail. Then, one has
\begin{align*}
\PP(F&-\EE F\geq r) \leq \exp\left({r\over\alpha}-\left(\beta+{r\over\alpha}\right)\log\left(1+{r\over \alpha\beta}\right) \right),\qquad r\geq 0,
\end{align*}
and
\begin{align*}
\PP(F&-\EE F\leq -r) \leq \exp\left(-{r\over\alpha}-\left(\beta-{r\over\alpha}\right)
\log\left(1-{r\over \alpha\beta}\right) \right),\qquad 0\leq r\leq\EE F,
\end{align*}
where
\begin{align}\label{eq:DefAlphaBeta}
\alpha=\lambda_{d-k}(M)\diam(W)^k\qquad\text{and}\qquad\beta = {p\over \lambda_{d-k}(M)}\sum\limits_{j=0}^{d-k}{\kappa_j\kappa_{d-j}\over{d\choose j}\kappa_d}V_j(W)V_{d-k-j}(M).
\end{align}
\end{corollary}

\begin{proof}
We apply Corollary \ref{cor:DilatedRotatedCylBase} and assume in addition that $R=1$ $\PP$-almost surely. In this case
\begin{align*}
{p\over m_{d-k}}\sum_{j=0}^{d-k}{\kappa_j\kappa_{d-j}\over{d\choose j}\kappa_d}V_j(W)V_{d-k-j}(M)\,\EE[R^{d-k-j}\Psi(\alpha\,s\,R^{d-k})] = \beta\,\Psi(\alpha\,s),
\end{align*}
and hence,
\begin{align*}
\PP(F&-\EE F\geq r) \leq \exp\left(\inf_{s\geq 0}\left(\beta\Psi(\alpha\,s)-rs\right)\right)=\exp\left(\inf_{s\geq 0}\left(\beta(e^{\alpha s}-\alpha s-1)-rs\right)\right).
\end{align*}
It is easy to verify that the infimum is attained at $s={1\over\alpha}\log(1+{r\over\alpha\beta})$. This gives
\begin{align*}
\PP(F&-\EE F\geq r) \leq \exp\Big({r\over\alpha}-\Big({r\over\alpha}+\beta\Big)\log\Big(1+{r\over\alpha\beta}\Big)\Big)
\end{align*}
for any $r\geq 0$. This completes the proof for the upper tail.

Similarly, consider the lower tail
\begin{align*}
\PP(F&-\EE F\leq -r) \leq \exp\left(\inf_{s\geq 0}\left(\beta\Psi(-\alpha\,s)-rs\right)\right)=\exp\left(\inf_{s\geq 0}\left(\beta(e^{-\alpha s}+\alpha s-1)-rs\right)\right).
\end{align*}
In case $r<\alpha\beta$ the infinum is attained at $s=-{1\over\alpha}\log(1-{r\over\alpha\beta})$ and the proof is completed. It just remains to be justified that $\alpha\beta \ge \EE F$ for any convex $M$ and $W$, and $k\ge 0$.

Due to the fact that intrinsic volumes are non-negative functionals on the family of convex bodies we conclude that
\begin{align*}
    \alpha\beta&= p\,\diam(W)^k\sum\limits_{j=0}^{d-k}{\kappa_j\kappa_{d-j}\over{d\choose j}\kappa_d}V_j(W)V_{d-k-j}(M)\\
    &\ge p\,\diam(W)^k{\kappa_{d-k}\kappa_{k}\over{d\choose d-k}\kappa_d}V_{d-k}(W)\\
    &=p\,\diam(W)^d{\kappa_{d-k}\kappa_{k}\over{d\choose d-k}\kappa_d}V_{d-k}(\widetilde{W}),
\end{align*}
where $\widetilde{W} = \diam(W)^{-1}\, W$ and, thus, $V_d(\widetilde{W})\leq 1$. From the isoperimetric inequality for intrinsic volumes of convex bodies (see, e.g., \cite[Equation (14.31)]{SW}) we conclude that
\begin{align*}
    {\kappa_k\over {d\choose d-k}}V_{d-k}(\widetilde{W})&\ge \kappa_{d}^{{k\over d}}V_d(\widetilde{W})^{1-{k\over d}}\ge \kappa_{d}^{{k\over d}}V_d(\widetilde{W}).
\end{align*}
Substituting this into above inequality we get
\begin{align*}
    \alpha\beta&\ge p\,\diam(W)^d{\kappa_{d-k}\over\kappa_d^{1-{k\over d}}}V_{d}(\widetilde{W})=\left({\kappa_{d-k}^d\over\kappa_d^{d-k}}\right)^{1/d}p\,V_d(W)=\left({\kappa_{d-k}^d\over\kappa_d^{d-k}}\right)^{1/d}\EE F.
\end{align*}
It remains to show that ${\kappa_{d-k}^d\over\kappa_d^{d-k}}\ge 1$, which is equivalent to
\begin{equation}\label{eq:02-08-2019}
\Gamma\left(1+{d\over 2}\right)^{1\over d}\ge \Gamma\left(1+{d-k\over 2}\right)^{1\over d-k}.
\end{equation}
However, this follows from the fact that the function $g(x):=\Gamma\left(1+{x\over 2}\right)^{1/x}$, $x>0$, is strictly increasing according to \cite[Theorem 1]{QiGuo}. This completes the proof.
\end{proof}
\subsection{Spherical windows}

Our general concentration inequality in Theorem \ref{thm:GeneralConcentrationInequ} simplifies if we assume the shape of our observation window $W$ to be spherical. More precisely, we assume that $W=B_R^d$ is a centred Euclidean ball of some fixed radius $R>0$. 

\begin{corollary}\label{cor:SphericalWindow}
	Let the general assumptions of Section \ref{SecVolume} prevail, and let $W=B_R^d$. Assuming that $\EE[V_j(\Xi)e^{s\lambda_{d-k}(\Xi)}]<\infty$ for some $s>0$ and all $j\in\{0,1,\ldots,d-k\}$ we have that, for $r\geq 0$,
	\begin{align*}
	\PP(F&-\EE F\geq r) \leq \exp\Big(\inf_{s\geq 0}\Big[{p\over m_{d-k}}\sum_{j=0}^{d-k}R^{d-k-j}\kappa_{d-k-j}\EE[V_j(\Xi)\Psi(s\lambda_{d-k}(\Xi)(2R)^k)]-rs\Big]\Big).
	\end{align*}
	Moreover, assuming that $\EE[V_j(\Xi)\lambda_{d-k}(\Xi)]<\infty$ for all $j\in\{0,1,\ldots,d-k\}$ we have that, for $0\leq r\leq\EE F$,
	\begin{align*}
	\PP(F&-\EE F\leq -r)\leq \exp\Big(\inf_{s\geq 0}\Big[{p\over m_{d-k}}\sum_{j=0}^{d-k}R^{d-k-j}\kappa_{d-k-j}\EE[V_j(\Xi)\Psi(-s\lambda_{d-k}(\Xi)(2R)^k)]-rs\Big]\Big).
	\end{align*}
\end{corollary}
\begin{proof}
	We have to analyze the term
	$$
	\EE[\lambda_{d-k}(P_{d-k}(\Theta^TW)+\Xi^*)\Psi(s\lambda_{d-k}(\Xi)\diam(W)^k)]
	$$
	appearing in Theorem\,\ref{thm:GeneralConcentrationInequ}, where now $W=B_R^d$. Since $P_{d-k}(\theta^TB_R^d)=B_R^{d-k}$ for any $\theta\in \SS\OO_{d,k}$ and since $\diam(B_R^d)=2R$ we have that
\begin{align*}
    \EE[\lambda_{d-k}(P_{d-k}(\Theta^TW)+\Xi^*)\Psi(s\lambda_{d-k}(\Xi)\diam(W)^k)] = \EE[\lambda_{d-k}(B_R^{d-k}+\Xi^*)\Psi(s\lambda_{d-k}(\Xi)(2R)^k)].
\end{align*}
    We are now in the position to apply Steiner's formula \cite[Equation (14.5)]{SW} in $\RR^{d-k}$. Together with Fubini's theorem and the reflection invariance of the intrinsic volumes this yields
\begin{align*}
    \EE[\lambda_{d-k}(B_R^{d-k}+\Xi^*)\Psi(s\lambda_{d-k}(\Xi)(2R)^k)] = \sum_{j=0}^{d-k}R^{d-k-j}\kappa_{d-k-j}\EE[V_j(\Xi)\Psi(s\lambda_{d-k}(\Xi)(2R)^k)].
\end{align*}
This proves the claim for the upper tail, the lower tail is similar.
\end{proof}

If in addition the typical cylinder base is spherical as well, the inequalities simplify further. We assume that $\PP$-almost surely $\Xi=B_\rho^{d-k}$ for some fixed $\rho>0$. Then
$$
V_j(B_\rho^{d-k}) = {\kappa_{d-k}\over\kappa_{d-k-j}}{d-k\choose j}\,\rho^j,\qquad\qquad j\in\{0,1,\ldots,d-k\}.
$$
Thus,
\begin{align*}
    &\sum_{j=0}^{d-k}R^{d-k-j}\kappa_{d-k-j}\EE[V_j(\Xi)\Psi(s\lambda_{d-k}(\Xi)(2R)^k)]\\
    &\qquad= \kappa_{d-k}\Psi(s\kappa_{d-k}\rho^{d-k}(2R)^k)\sum_{j=0}^{d-k}{d-k\choose j}R^{d-k-j}\rho^j\\
    &\qquad= \kappa_{d-k}\Psi(s\kappa_{d-k}\rho^{d-k}(2R)^k)\,R^{d-k}\Big(1+{\rho\over R}\Big)^{d-k}.
\end{align*}
Putting
$$
a:=\kappa_{d-k}\rho^{d-k}(2R)^k \qquad\text{and}\qquad b:={p\over m_{d-k}}\kappa_{d-k}R^{d-k}(1+{\rho\over R})^{d-k}
$$
it is easy to check that the function $f(s) := b\big(e^{as}-as-1\big)-rs$ attains its infimum over the set $\{s\geq 0\}$ at $s={1\over a}\log(1+{r\over ab})$. Together with the previous corollary this yields the following result.

\begin{corollary}
If $W=B_R^d$ and $\PP$-almost surely $\Xi=B_\varrho^d$ for some fixed $R,\varrho\in(0,\infty)$ then
\begin{align*}
\PP(F-\EE F\geq r) \leq \exp\Big({r\over a}-\Big(b+{r\over a}\Big)\log\Big(1+{r\over ab}\Big)\Big),\qquad r\geq 0,
\end{align*}
and
$$
\PP(F-\EE F\leq -r)\leq \exp\Big(-{r\over a}-\Big(b-{r\over a}\Big)\log\Big(1-{r\over ab}\Big)\Big),\qquad 0\leq r\leq \EE F.
$$
\end{corollary}

\subsection{Discussion}\label{subsec:Discussion}

Let us discuss the quality of the bounds we derived in the previous sections, where we restrict our attention to Corollary \ref{cor:RotatedCylBase}. Since
$$
-{r\over\alpha}-\left(\beta-{r\over\alpha}\right)\log \left(1-{r\over \alpha\beta}\right) \leq -{r^2\over 2\alpha^2\beta},\qquad 0\leq r\leq \alpha\beta,
$$
we infer for the lower tail that
$$
\PP(F-\EE F\leq -r) \leq \exp\Big(-{r^2\over 2\alpha^2\beta}\Big),\qquad 0\leq r\leq\EE F.
$$
Next, we discuss the upper tail. For $r\to\infty$ we obtain that
\begin{align}\label{eq:UpperTailRLogR}
\PP(F-\EE F\geq r)\leq\exp(-\boldsymbol{\Theta}(r\log r)),
\end{align}
where we recall that $\boldsymbol{\Theta}(r\log r)$ denotes a quantity in $\boldsymbol{O}(r\log r)\cap\boldsymbol{\Omega}(r\log r)$ and our window $W$ does not depend on $r$. 

Although no concentration inequality for $F$ is explicitly available in the literature, such an inequality easily follows from the sharp cumulant estimates carried out in \cite{HeinrichSpiessCLTVolume}. In fact, applying \cite[Lemma 2.4]{Saulis} to these estimates yields a bound for the upper tail of the form
$$
\PP(F-\EE F\geq r)\leq\exp(-\boldsymbol{\Theta}(r)),
$$
as $r\to\infty$. Clearly, this is weaker than the bound \eqref{eq:UpperTailRLogR} we got. Moreover, if $X$ is a Poisson random variable with parameter $\lambda>0$ then 
$$
\PP(X-\EE X\geq r) \leq \exp\Big(r-(\lambda+r)\log\Big(1+{r\over\lambda}\Big)\Big),\qquad r\geq 0,
$$
which is asymptotically tight, as $r\to\infty$, up to a factor $(2\pi(\lambda+r))^{-1/2}$, see \cite{Houdre}. A comparison with Corollary \ref{cor:RotatedCylBase} thus shows that, for a fixed window $W$, our bound for the upper tail is essentially of the same order as the one for a Poisson random variable. This leads us to the conclusion that the exponential order in $r$ of our bound is presumably optimal.

It is a remarkable observation that the bound \eqref{eq:UpperTailRLogR} is of the same order as the one for the stationary Boolean model in $\RR^d$ discussed in \cite{GieringerDissertation,GieringerLast}. This might be somewhat surprising, since the correlation structure of the union set of a stationary Boolean model and of a stationary Poisson cylinder process are quite different. In fact, while for $k=0$ the functional $F$ is of volume-order, for $k\geq 1$ the random set $Z$ admits strong long-range correlations, which are propagated by the infinitely long cylinders over the whole space. This is also well reflected, for example, by the growth of the variance of the total volume of $Z$ for a sequence of growing windows $W_r=rW$, $r>0$. For example, it is known from \cite{HeinrichSpiessCLTVolume} that the variance of $\lambda_d(Z\cap W_r)$ is of order $r^{d+k}$, which for $k\geq 1$ is strictly larger than the volume-order $r^d$.
To relate this discussion to our inequalities, we shall now consider the case when the window is growing with $r$. In fact, we consider the situation in which the window is of the form $r^{1/d}W$ for fixed convex body $W\subset\RR^d$. This choice corresponds to a linear growth of the volume of the window with $r$. Moreover, we assume that the typical cylinder base arises from a fixed convex body $M\subset\RR^{d-k}$ by a uniform random rotation in $\mathbb{R}^{d-k}$. Then, recalling \eqref{eq:DefAlphaBeta}, we have that
\begin{align*}
    \alpha=2^kV_{d-k}(M)r^{k/d}\qquad\text{and}\qquad\beta = {p\over m_{d-k}}\sum\limits_{j=0}^{d-k}\kappa_{j}r^{j/d}V_{d-k-j}(M).
\end{align*}
We note that, as $r\to\infty$, $\alpha=\boldsymbol{\Theta}(r^{k/d})$, while $\beta=\boldsymbol{\Theta}(r^{(d-k)/d})$. Plugging this into Corollary \ref{cor:RotatedCylBase} we find that
\begin{align}\label{eq:24-06-19}
\mathbb{P}(F-\mathbb{E} F\geq r) \leq \exp(-\boldsymbol{\Theta}(r^{1-k/d})),
\end{align}
as $r\to\infty$. This bound clearly reflects the dependence on the dimension parameter $k$ and also shows that the bound becomes weaker the bigger $k$ is chosen.

\section{Concentration inequalities for intrinsic volumes}\label{sec:IntVol}

The purpose of this section is to prove a concentration inequality for the intrinsic volumes associated with the union set $Z$ of a stationary and isotropic Poisson process of $k$-cylinders in $\RR^d$. For this we assume in this section that the typical cylinder base $\Xi$ is convex $\PP$-almost surely and also that the base-direction distribution $\QQ$ is rotation invariant. In view of Lemma \ref{lem:isotropy} and the following discussion, this implies that $Z$ is a stationary and isotropic random closed set. We also assume that the window $W$ is convex.

\subsection{Mean value formulas}

The proof of our tail bounds relies on the general concentration inequalities from Section \ref{sec:GeneralConcentration} as well as on a mean value formula for the intrinsic volumes of $Z\cap W$. While such formulas are well known for the Boolean model (see, e.g., \cite[Theorem 9.1.3]{SW}), we were not able to locate a corresponding result for the union set of Poisson cylinder processes in the existing literature (for the closest results in this direction we refer to \cite[Section 5]{Hoffmann} and \cite[Section 7]{Weil90}). The purpose of this section is to provide such formulas under the assumption that $\QQ$ is rotation invariant. In particular, this assumption allows us to use the principal kinematic formula for cylinders from \cite[Chapter 6.3]{SW}.

\begin{proposition}\label{lm:IntrVolumesFraction}
Let $W\subset\RR^d$ be convex body with $V_d(W)>0$ and let $0\leq j=:j_0\leq d$ be some integer. Suppose that $\Xi$ is convex $\PP$-almost surely and that $\QQ$ is rotation invariant. Assume further that $m_i:=\EE V_i(\Xi)<\infty$ for $j-k\leq i\leq d-k$. Then
$$
\EE V_{j}(Z\cap W) = \sum\limits_{\ell=1}^{\infty}{(-1)^{\ell-1}\gamma^{\ell}\over \ell!}\sum\limits_{j_1=j_0}^{\min(d, d+j_0-k)}\cdots\!\!\!\sum\limits_{j_{\ell}=j_{\ell-1}}^{\min(d, d+j_{\ell-1}-k)}\!\!\!\!\!c_{j_0}^{j_{\ell}}V_{j_{\ell}}(W)\prod\limits_{i=1}^{\ell}c^{d+j_{i-1}-j_i}_{d}m_{d-k+j_{i-1}-j_i},
$$
where $c_{r}^{p}={p!\kappa_{p}\over r!\kappa_{r}}$. If additionally $j\ge k$, then
\begin{align*}
\EE V_{j}(Z\cap W) &= V_{j}(W)\left(1-e^{-\gamma\,m_{d-k}}\right)\\
&\qquad\qquad-e^{-\gamma\,m_{d-k}}\sum\limits_{m=1}^{d-j}c_{j}^{m+j}V_{m+j}(W)\sum\limits_{p=1}^m{(-1)^{p}\gamma^p\over p!}\sum\limits_{\substack{q_1,\ldots, q_p>0\\ q_1+\ldots+q_p=m}}\prod\limits_{i=1}^p c_{d}^{d-q_i}m_{d-k-q_i},
\end{align*}
where the empty sum is interpreted as zero.
\end{proposition}

\begin{remark}\rm 
We emphasize that for $j=d$ and $j=d-1$ the formula in Proposition \ref{lm:IntrVolumesFraction} can considerably be simplified. In fact, we have that
$$
\EE V_d(Z\cap W) = V_d(W)(1-e^{-\gamma m_{d-k}}),
$$
see \eqref{eq:ExpectationF}, and
$$
\EE V_{d-1}(Z\cap W) = \gamma\,V_d(W)\,m_{d-k-1}\,e^{-\gamma\,m_{d-k}}+V_{d-1}(W)(1-e^{-\gamma\,m_{d-k}}).
$$
\end{remark}

\begin{proof}[Proof of Proposition \ref{lm:IntrVolumesFraction}]
By definition of $Z$, the inclusion-exclusion principle and the multivariate Mecke formula (see \cite[Theorem 4.4]{LP}) we have that
\begin{align*}
\EE V_{j}(Z\cap W) &= \EE V_{j}\bigg(\bigcup_{(x,\theta,K)\in\xi}Z(x,\theta,K)\cap W\bigg)\\
&=\sum_{\ell=1}^\infty{(-1)^{\ell-1}\over\ell!}\,\gamma^\ell\int\limits_{\MM_{d,k}^\ell}\int\limits_{(\RR^{d-k})^\ell}V_{j}(Z(x_1,\theta_1,K_1)\cap\ldots\cap Z(x_\ell,\theta_\ell,K_\ell)\cap W)\\
&\qquad\qquad\qquad\qquad\qquad\qquad\times\lambda_{d-k}^\ell(\dint(x_1,\ldots,x_\ell))\,\QQ^\ell(\dint((\theta_1,K_1),\ldots,(\theta_\ell,K_\ell))).
\end{align*}
To evaluate the $\ell$-fold integral over $\RR^{d-k}$ we make use of the following principal kinematic formula for cylinders, which can be found in \cite[Corollary 6.3.1]{SW}. Namely, for fixed $(\theta,K)\in\MM_{d,k}$ one has that
\begin{align*}
\int\limits_{\SS\OO_d}\int\limits_{\RR^{d-k}}V_{j}(\varrho Z(x,\theta,K)\cap W)\,\lambda_{d-k}(\dint x)\nu_d(\dint \varrho) = \sum\limits_{p=j}^{\min(d, d+j-k)}c_{j}^{p}c_{d}^{d-p+j}V_{p}(W)V_{d-k+j-p}(K),
\end{align*}
where $\nu_d$ is the unique rotationally invariant Haar probability measure on $\SS\OO_d$.
A recursive application of this integral formula and Fubini's theorem yields that, for fixed $\ell\in\NN$ and $(\theta_1,K_1),\ldots,(\theta_\ell,K_\ell)\in\MM_{d,k}$,
\begin{align*}
&\int\limits_{(\SS\OO_{d})^\ell}\int\limits_{(\RR^{d-k})^\ell} V_{j}(\varrho_1 Z(x_1,\theta_1,K_1)\cap\ldots \cap \varrho_{\ell} Z(x_{\ell},\theta_{\ell},K_{\ell})\cap W)\\
&\hspace{5cm}\times\lambda_{d-k}^\ell(\dint(x_1,\ldots,x_\ell))\,\nu_{d}^\ell(\dint(\varrho_1,\ldots,\varrho_\ell))\\
&\qquad= \sum\limits_{j_1=j_0}^{\min(d, d+j_0-k)}\cdots\sum\limits_{j_{\ell}=j_{\ell-1}}^{\min(d, d+j_{\ell-1}-k)}c_{j_0}^{j_{\ell}}V_{j_{\ell}}(W)\prod\limits_{i=1}^{\ell}c^{d+j_{i-1}-j_i}_{d}V_{d-k+j_{i-1}-j_i}(K_i),
\end{align*}
where we recall that $j_0=j$.
Thus, from the assumed rotational invariance of $\QQ$ and Fubini's theorem we conclude
\begin{align*}
    &\EE V_{j}(Z\cap W)\\
    &\qquad=\sum_{\ell=1}^\infty{(-1)^{\ell-1}\over\ell!}\,\gamma^\ell\int\limits_{\MM_{d,k}^\ell}\int\limits_{(\SS\OO_{d})^\ell}\int\limits_{(\RR^{d-k})^\ell}V_{j}(\varrho_1Z(x_1,\theta_1,K_1)\cap\ldots\cap \varrho_\ell Z(x_\ell,\theta_\ell,K_\ell)\cap W)\\
    &\qquad\qquad\qquad\qquad\times\lambda_{d-k}^\ell(\dint(x_1,\ldots,x_\ell))\,\nu_{d}^\ell(\dint(\varrho_1,\ldots,\varrho_\ell))\,\QQ^\ell(\dint((\theta_1,K_1),\ldots,(\theta_\ell,K_\ell)))\\
    &\qquad=\sum_{\ell=1}^\infty{(-1)^{\ell-1}\over\ell!}\,\gamma^\ell\sum\limits_{j_1=j_0}^{\min(d, d+j_0-k)}\cdots\sum\limits_{j_{\ell}=j_{\ell-1}}^{\min(d, d+j_{\ell-1}-k)}c_{j_0}^{j_{\ell}}V_{j_{\ell}}(W)\prod\limits_{i=1}^{\ell}c^{d+j_{i-1}-j_i}_{d}m_{d-k+j_{i-1}-j_i}.
\end{align*}
This proves the first claim.

If $j\ge k$, the above formula can be simplified further. To this end, let us introduce the notation $q_i:=j_i-j_{i-1}$. Then  $\sum\limits_{i=1}^{\ell}q_i=j_\ell-j_0$ and we obtain that
\begin{align*}
&\EE V_{j}(Z\cap W)\\
&=\sum_{\ell=1}^\infty{(-1)^{\ell-1}\over\ell!}\,\gamma^\ell\sum\limits_{j_1=j_0}^{d}\cdots\sum\limits_{j_{\ell}=j_{\ell-1}}^{d}c_{j_0}^{j_{\ell}}V_{j_{\ell}}(W)\prod\limits_{i=1}^{\ell}c^{d+j_{i-1}-j_i}_{d}m_{d-k+j_{i-1}-j_i}\displaybreak\\
&=\sum_{\ell=1}^\infty{(-1)^{\ell-1}\over\ell!}\,\gamma^\ell\sum\limits_{q_1=0}^{d-j_0}\sum\limits_{q_2=0}^{d-j_0-q_1}\cdots\sum\limits_{q_\ell=0}^{d-j_0-q_1-\ldots-q_{\ell-1}}c_{j_0}^{q_1+\ldots+q_\ell +j_0}V_{q_1+\ldots+q_\ell+j_0}(W)\prod\limits_{i=1}^{\ell}c^{d-q_i}_{d}m_{d-k-q_i}\\
&=V_j(W)(1-e^{-\gamma m_{d-k}})+\sum\limits_{m=1}^{d-j}c_{j}^{m+j}V_{m+j}(W)\,S_m
\end{align*}
with
$$
S_m:=\sum_{\ell=1}^\infty{(-1)^{\ell-1}\over\ell!}\,\gamma^\ell\sum\limits_{\substack{q_1,\ldots,q_{\ell}\ge 0 \\ q_1+\ldots+q_{\ell}=m}}\prod\limits_{i=1}^{\ell}c^{d-q_i}_{d}m_{d-k-q_i}.
$$
Assume further that $\ell=r+p$, $1\leq p\leq m$, $r\in\{0,1,2,\ldots\}$ and $q_1,\ldots, q_p\in\mathbb{N}$, $q_{p+1}=\ldots=q_\ell=0$. Then the infinite sum $S$ can be evaluated explicitly. Indeed, we have that
\begin{align*}
S_m&=\sum_{\ell=1}^\infty{(-1)^{\ell-1}\over\ell!}\,\gamma^\ell\sum\limits_{\substack{q_1,\ldots,q_{\ell}\ge 0 \\ q_1+\ldots+q_{\ell}=m}}\prod\limits_{i=1}^{\ell}c^{d-q_i}_{d}m_{d-k-q_i}\\
&=\sum\limits_{p=1}^m\sum\limits_{r=0}^{\infty}{(-1)^{r+p-1}\gamma^{r+p}\over (r+p)!}\binom{r+p}{r}m_{d-k}^r\sum\limits_{\substack{q_1,\ldots,q_{p}> 0 \\ q_1+\ldots+q_{p}=m}}\prod\limits_{i=1}^{p}c^{d-q_i}_{d}m_{d-k-q_i}\\
&=-e^{-\gamma m_{d-k}}\sum\limits_{p=1}^m{(-1)^{p}\gamma^{p}\over p!}\sum\limits_{\substack{q_1,\ldots,q_{p}> 0 \\ q_1+\ldots+q_{p}=m}}\prod\limits_{i=1}^{p}c^{d-q_i}_{d}m_{d-k-q_i}
\end{align*}
and the proof is complete.
\end{proof}

\subsection{Concentration inequality}

For fixed $j\in\{0,1,\ldots,d\}$ we consider the Poisson functional
$$
F_j := V_j(Z\cap W).
$$
We start by dealing with the first-order difference operator of $F_j$. Due to the additivity of the intrinsic volumes, for $(\lambda_{d-k}\otimes\QQ)$-almost all $(x,\theta,K)\in\RR^{d-k}\times\MM_{d,k}$ we have that
\begin{align*}
D_{(x,\theta,K)}F_j &= V_{j}\left((Z\cup Z(x,\theta,K))\cap W\right) - V_{j}(Z\cap W)\\
&=V_{j}\left( Z(x,\theta,K)\cap W\right) - V_{j}(Z\cap Z(x,\theta,K)\cap W)
\end{align*}
holds $\PP$-almost surely.

In order to derive a bound for the upper tail and the lower tail of the functional $F_j$ we will apply the technique already used in Section \ref{SecVolume} for the case of the volume. For this we need to make sure that the conditions of Lemma \ref{lem:GieringerLast} hold for $F_j$.

\begin{lemma}\label{lm:AssumptionsIntrVolume}
For any $j\in\{0,1,\ldots,d\}$ we have that $F_j\in L^2(\PP)$, $DF_j\in L^2(\PP\otimes \lambda_{d-k}\otimes\QQ)$ and $s_{F_j}=s_{F_j}^{({\rm lt})}=\infty$.
\end{lemma}

\begin{proof}
Since the intrinsic volumes $V_j$ are non-negative and monotone under set inclusion on the family of convex bodies we have that $\PP$-almost surely 
\[
D_{(x,\theta,K)}F_j\leq V_{j}\left( Z(x,\theta,K)\cap W\right)\leq V_{j}\left(W\right)
\]
for all $(x,\theta,K)\in\RR^{d-k}\times\MM_{d,k}$. The rest of the proof is now analogous to the proof of Lemma \ref{lem:AssumptionsVolume}.
\end{proof}

\begin{theorem}\label{thm:ConcentrationInequalityIntrVol}
Let $W\subset\RR^d$ be a convex body with $V_d(W)>0$, $\Xi$ be convex $\PP$-almost surely and assume that $\QQ$ is rotation invariant. Also, suppose that $j\ge k$ and $m_i\in(0,\infty)$ for all $j-k\leq i\leq d-k$. Then, for all $r\geq 0$, one has that
\begin{align*}
\PP(F_j-\EE F_j\geq r) &\leq \exp\Big(\inf_{s\geq 0}\Big(\EE\Big[V_{d-k}(P_{d-k}(\Theta^TW)+\Xi^*)\Psi\Big(s\sum\limits_{i=j-k}^{\min\{d-k,j\}}\diam(W)^{j-i}\textstyle{k\choose j-i}V_i(\Xi)\Big)\\
     &\qquad\qquad\qquad\times\sum\limits_{m=0}^{d-j}\beta_m\Big(\sum\limits_{i=j-k}^{\min\{d-k,j\}}\diam(W)^{j-i}\textstyle{k\choose j-i}V_i(\Xi)\Big)^{m/j}\Big]-rs\Big)\Big),
\end{align*}
and for $0\leq r\leq\EE F_j$ one has that
\begin{align*}
\PP(F_j-\EE F_j\leq -r) &\leq  \exp\Big(\inf_{s\geq 0}\Big(\EE\Big[V_{d-k}(P_{d-k}(\Theta^TW)+\Xi^*)\Psi\Big(-s\sum\limits_{i=j-k}^{\min\{d-k,j\}}\diam(W)^{j-i}\textstyle{k\choose j-i}V_i(\Xi)\Big)\\
     &\qquad\qquad\qquad\times\sum\limits_{m=0}^{d-j}\beta_m\Big(\sum\limits_{i=j-k}^{\min\{d-k,j\}}\diam(W)^{j-i}\textstyle{k\choose j-i}V_i(\Xi)\Big)^{m/j}\Big]-rs\Big)\Big),
\end{align*}
where $\Psi(x)=e^x-x-1$, $x\in\RR$, and
\begin{align*}
    \beta_0&:={1-e^{-\gamma m_{d-k}}\over m_{d-k}},\qquad \beta_1=0,\\
    \beta_m&:={\kappa_{d-j}^{1+m/j}{d\choose j+m}c_{j}^{m+j}\over \kappa_d^{m/j}\kappa_{d-j-m}{d\choose j}^{1+m/j}}\,\sum\limits_{p=1}^{\lfloor {m\over 2}\rfloor}m_{d-k}^{-2p-1}\Big(1-e^{-\gamma m_{d-k}}\sum\limits_{i=0}^{2p}{(\gamma m_{d-k})^{2p-i} \over (2p-i)!}\Big)\\
    &\hspace{8cm}\times\sum\limits_{\substack{q_1,\ldots, q_{2p}>0\\ q_1+\ldots+q_{2p}=m}}\prod\limits_{i=1}^{2p} c_{d}^{d-q_i}m_{d-k-q_i}
\end{align*}
for $m\in\{2,\ldots,d-j\}$. 
\end{theorem}

\begin{remark}\rm 
\begin{itemize}
    \item[(i)]
We specialize the result of Theorem \ref{thm:ConcentrationInequalityIntrVol} for $j=d$ and $j=d-1$, where the concentration inequality takes a more simple form. For simplicity, we restrict ourselves to the bound for the upper tail. If $j=d$ we obtain, for $r\geq 0$,
\begin{align*}
    \PP(F_d-\EE F_d\geq r) \leq \exp\Big(\inf_{s\geq 0}\Big({p\over m_{d-k}}\EE\Big[V_{d-k}(P_{d-k}(\Theta^TW)+\Xi^*)\Psi(s\diam(W)^kV_{d-k}(\Xi))\Big]-rs\Big)\Big),
\end{align*}
which is precisely the bound we derived in Section \ref{SecVolume} under more general conditions, since $V_{d-k}(K)=\lambda_{d-k}(K)$ for a convex body $K\subset\RR^{d-k}$. Moreover, choosing $j=d-1$ we obtain, again for $r\geq 0$,
\begin{align*}
    \PP(F_{d-1}-\EE F_{d-1}\geq r) &\leq \exp\Big(\inf_{s\geq 0}\Big({p\over m_{d-k}}\EE\Big[V_{d-k}(P_{d-k}(\Theta^TW)+\Xi^*)\\
    &\qquad\times\Psi\big(s\diam(W)^{k-1}[\diam(W)V_{d-k-1}(\Xi)+kV_{d-k}(\Xi)]\big)\Big]-rs\Big)\Big).
\end{align*}

\item[(ii)]
Taking $k=0$, which corresponds to the Boolean model, and $j=d-1$ we deduce that
$$
\PP(F_{d-1}-\EE F_{d-1}\geq r) \leq \exp\Big(\inf_{s\geq 0}\Big({p\over m_d}\EE\big[V_d(W+\Xi^*)\Psi(sV_{d-1}(\Xi))\big]-rs\Big)\Big),\qquad r\geq 0,
$$
which should be compared to the corresponding inequality \eqref{eq:ConcBMVolume} for $F_d$. Note that the reason behind this simple form is the fact that the constant $\beta_1$ in Theorem \ref{thm:ConcentrationInequalityIntrVol} is equal to zero. Since this is not the case for $\beta_m$ with $m \in \{2,\ldots,d-j\}$, the resulting inequalities become more involved. In fact, for $j\in\{0,1,\ldots,d-1\}$ we have that
\begin{align*}
    \PP(F_j-\EE F_j\geq r) \leq \exp\Big(\inf_{s\geq 0}\Big(\EE\Big[V_{d}(W+\Xi^*)\Psi(sV_{j}(\Xi))\sum_{m=0}^{d-j}\beta_m V_j(\Xi)^{m/j}\Big]-rs\Big)\Big),\qquad r\geq 0.
\end{align*}
\end{itemize}
\end{remark}

\begin{proof}[Proof of Theorem \ref{thm:ConcentrationInequalityIntrVol}]
As in the proof of Theorem \ref{thm:GeneralConcentrationInequ} we start by deriving an upper bound for the function $V_{F_j}(s)$ defined by \eqref{eq:DefVF}. Considering the term $T_t$, $t\in[0,1]$, and applying Proposition \ref{lm:IntrVolumesFraction} we have that, putting $\Lambda:=\gamma\,\lambda_{d-k}\otimes\QQ$, 
\begin{align*}
T_t &\leq V_{j}(Z(x,\theta,K)\cap W) - \int\limits_{{\bf N}}V_{j}(Z(\mu)\cap Z(x,\theta,K)\cap W)\,\Pi_{(1-t)\Lambda}(\dint\mu)\\
&=e^{-\gamma(1-t)\,m_{d-k}}\bigg(V_j(Z(x,\theta,K)\cap W)\\
&\qquad\qquad+\sum\limits_{m=1}^{d-j}c_{j}^{m+j}V_{m+j}(Z(x,\theta,K)\cap W)\sum\limits_{s=1}^m{(-\gamma(1-t))^{s}\over s!}\sum\limits_{\substack{q_1,\ldots, q_s>0\\ q_1+\ldots+q_s=m}}\prod\limits_{i=1}^s c_{d}^{d-q_i}m_{d-k-q_i}\bigg),
\end{align*}
where for $\mu\in\textbf{N}$, $Z(\mu)$ stands for the union set induced by $\mu$.
Hence,
\begin{align*}
\int\limits_0^1 T_t\,\dint t &= V_j(Z(x,\theta,K)\cap W)I_0+\sum\limits_{m=1}^{d-j}c_{j}^{m+j}V_{m+j}(Z(x,\theta,K)\cap W)\\
&\hspace{5cm}\times\sum\limits_{p=1}^m I_p\gamma^p\sum\limits_{\substack{q_1,\ldots, q_p>0\\ q_1+\ldots+q_p=m}}\prod\limits_{i=1}^p c_{d}^{d-q_i}m_{d-k-q_i},
\end{align*}
where
\begin{align*}
I_p:&={1\over p!}\int\limits_0^1 (t-1)^{p}e^{\gamma(t-1)\,m_{d-k}}\,\dint t={(-1)^p\over (\gamma m_{d-k})^{p+1}}\left(1-e^{-\gamma m_{d-k}}\sum\limits_{i=0}^p{(\gamma m_{d-k})^{p-i} \over (p-i)!}\right).
\end{align*}
Let us introduce the following additional notation in order to simplify our subsequent computations:
\begin{align*}
    \alpha_0&:={1-e^{-\gamma m_{d-k}}\over m_{d-k}},\qquad \alpha_1=0,\\
    \alpha_m&:=c_{j}^{m+j}\sum\limits_{p=1}^{\lfloor {m\over 2}\rfloor}m_{d-k}^{-2p-1}\left(1-e^{-\gamma m_{d-k}}\sum\limits_{i=0}^{2p}{(\gamma m_{d-k})^{2p-i} \over (2p-i)!}\right)\sum\limits_{\substack{q_1,\ldots, q_{2p}>0\\ q_1+\ldots+q_{2p}=m}}\prod\limits_{i=1}^{2p} c_{d}^{d-q_i}m_{d-k-q_i}
\end{align*}
for $m\geq 1$.
Using the notation above and applying Fubini's theorem and the fact that intrinsic volumes are non-negative functionals on the family of convex bodies we conclude that
\begin{align*}
V_F(s) \leq v(s),\qquad s\geq 0,
\end{align*}
with $v(s)$ given by
\begin{align*}
v(s)=\int\limits_{\MM_{d,k}}\int\limits_{\RR^{d-k}}\left(e^{sV_j(Z(x,\theta,K)\cap W)}-1\right)\sum\limits_{m=0}^{d-j}\alpha_mV_{m+j}(Z(x,\theta,K)\cap W)\,\lambda_{d-k}(\dint x)\,\QQ(\dint(\theta,K)).
\end{align*}
In the next step we investigate the integral
$$
w(s):=\int\limits_{0}^s v(u)\,\dint u.
$$
For that purpose we notice that in the definition of $w(s)$ we can multiply the integrand with the indicator function ${\bf 1}\{Z(x,\theta,K)\cap {\rm int}(W)\neq\varnothing\}$. In fact, $Z(x,\theta,K)\cap {\rm int}(W)\neq\varnothing$ is equivalent to $V_j(Z(x,\theta,K)\cap {\rm int}(W))>0$ and, additionally,
$$
\lambda_{d-k}(\{x\in\RR^{d-k}:Z(x,\theta,K)\cap W\neq\varnothing\text{ and }Z(x,\theta,K)\cap {\rm int}(W)=\varnothing\})=0
$$
holds by our convexity assumption on the cylinder bases $K$. We can thus write 
\begin{align*}
    w(s):&=\int\limits_{0}^s v(u)\,\dint u=\int\limits_{\MM_{d,k}}\int\limits_{\RR^{d-k}}\left [e^{sV_j(Z(x,\theta,K)\cap W)}-sV_j(Z(x,\theta,K)\cap W)-1\right]\\
    &\qquad\qquad\times \sum\limits_{m=0}^{d-j}\alpha_m{V_{m+j}(Z(x,\theta,K)\cap W)\over V_{j}(Z(x,\theta,K)\cap W)}\,{\bf 1}\{Z(x,\theta,K)\cap {\rm int}(W)\neq\varnothing\}\,\lambda_{d-k}(\dint x)\,\QQ(\dint(\theta,K))\\
    &=\int\limits_{\MM_{d,k}}\int\limits_{\RR^{d-k}}\Psi(sV_j(Z(x,\theta,K)\cap W))\sum\limits_{m=0}^{d-j}\alpha_m{V_{m+j}(Z(x,\theta,K)\cap W)\over V_{j}(Z(x,\theta,K)\cap W)}\\
    &\qquad\qquad\qquad\qquad\times{\bf 1}\{Z(x,\theta,K)\cap {\rm int}(W)\neq\varnothing\}\,\lambda_{d-k}(\dint x)\,\QQ(\dint(\theta,K)),
\end{align*}
where $\Psi(x)=e^x-x-1$, $x\in\RR$. Let us note here that the function $\Psi(x)$ is increasing for $x\ge 0$ and that all coefficients $\alpha_m$ are non-negative. Thus, the integrand is an increasing function in $V_{m+j}(Z(x,\theta,K)\cap W)$, $1\leq m\leq d-j$. From the isoperimetric inequalities for intrinsic volumes of convex bodies (see, e.g., \cite[Equation (14.31)]{SW}) we deduce that, for $(x,\theta,K)\in\MM_{d,k}$,
\begin{align*}
    V_{m+j}(Z(x,\theta,K)\cap W)\leq {\kappa_{d-j}^{1+m/j}{d\choose j+m}\over \kappa_d^{m/j}\kappa_{d-j-m}{d\choose j}^{1+m/j}}V_{j}(Z(x,\theta,K)\cap W)^{m/j+1}.
\end{align*}
Provided that $Z(x,\theta,K)\cap {\rm int}(W)\neq\varnothing$, this implies
\begin{equation}\label{eq:final}
\begin{split}
    w(s) &\leq\int\limits_{\MM_{d,k}}\int\limits_{\RR^{d-k}}\Psi(sV_j(Z(x,\theta,K)\cap W))\sum\limits_{m=0}^{d-j}\beta_mV_{j}(Z(x,\theta,K)\cap W)^{m/j}\,\lambda_{d-k}(\dint x)\,\QQ(\dint(\theta,K)),
\end{split}
\end{equation}
with the coefficients $\beta_0,\ldots,\beta_{d-j}$ given by
$$
\beta_m:={\alpha_m\kappa_{d-j}^{1+m/j}{d\choose j+m}\over \kappa_d^{m/j}\kappa_{d-j-m}{d\choose j}^{1+m/j}},\qquad m\in\{0,1,\ldots,d-j\}.
$$
Note that we can from now on omit the indicator function that $Z(x,\theta,K)\cap {\rm int}(W)\neq\varnothing$. Using again the fact that the window $W$ as well as the cylinder bases $K$ are convex and that the intrinsic volumes are monotone under set inclusion on the family of convex bodies we get
\[
V_j(Z(x,\theta,K)\cap W)\leq V_j(K+\diam(W)\,C_k),
\]
where $C_k\subset E_k$ denotes the $k$-dimensional unit cube. Applying now \cite[Lemma 14.2.1]{SW} we conclude that
\[
V_j(Z(x,\theta,K)\cap W)\leq \sum\limits_{i=j-k}^{\min\{d-k,j\}}\diam(W)^{j-i}{k\choose j-i}V_i(K).
\]
Substituting this into \eqref{eq:final} and proceeding as in the proof of Theorem \ref{thm:GeneralConcentrationInequ} we obtain
\begin{align*}
     w(s)&\leq \int\limits_{\MM_{d,k}}V_{d-k}(P_{d-k}(\theta^TW)+K^*)\,\Psi\left(s\sum\limits_{i=j-k}^{\min\{d-k,j\}}\diam(W)^{j-i}{k\choose j-i}V_i(K)\right)\\
     &\qquad\qquad\times\sum\limits_{m=0}^{d-j}\beta_m\left(\sum\limits_{i=j-k}^{\min\{d-k,j\}}\diam(W)^{j-i}{k\choose j-i}V_i(K)\right)^{m/j}\,\QQ(\dint(\theta,K)).
\end{align*}
This completes the proof for the upper tail, the proof for the lower tail is similar.
\end{proof}

\subsection{The special case of randomly rotated cylinder bases}

The result of Theorem \ref{thm:ConcentrationInequalityIntrVol} can be simplified further if we additionally assume that the cylinder bases are random rotations of a fixed convex body $M\subset\RR^{d-k}$ and the direction $\Theta$ of the typical cylinder base is uniformly distributed on $\SS\OO_{d,k}$ according to the unique rotationally invariant Haar probability measure $\nu_{d,k}$ on $\SS\OO_{d,k}$, independently of $U$ (recall Sections \ref{subsec:RotationDilation} and \ref{subsec:Rotation}). More explicitly, this means that $\Xi=UM$, where $U\in SO_{d-k}$ is a uniform random rotation in $\RR^{d-k}$, and that $\Xi$ and $\Theta$ are independent.

\begin{corollary}
Under the assumptions just described we have that, for all integers $k\leq j\leq d$,
\begin{align*}
\PP(F_j&-\EE F_j\geq r) \leq \exp\left({r\over\alpha}-\left(\beta+{r\over\alpha}\right)\log\left(1+{r\over \alpha\beta}\right) \right),\qquad r\geq 0,
\end{align*}
and
\begin{align*}
\PP(F_j&-\EE F_j\leq -r) \leq \exp\left(-{r\over\alpha}-\left(\beta-{r\over\alpha}\right)\log\left(1-{r\over \alpha\beta}\right) \right),\qquad 0\leq r\leq\EE F_j,
\end{align*}
where $\alpha=\sum\limits_{i=j-k}^{\min\{d-k,j\}}\diam(W)^{j-i}\textstyle{k\choose j-i}V_i(M)$ and $\beta = \sum\limits_{i=0}^{d-k}{\kappa_i\kappa_{d-i}\over{d\choose i}\kappa_d}V_i(W)V_{d-k-i}(M)\sum\limits_{m=0}^{d-j}\beta_m\alpha^{m/j}$, with $\beta_m$ defined as in Theorem \ref{thm:ConcentrationInequalityIntrVol}.
\end{corollary}

\begin{proof}
The proof is analogous to the proof of Corollary \ref{cor:DilatedRotatedCylBase} and Corollary \ref{cor:RotatedCylBase} due to invariance of intrinsic volumes under rotations.
\end{proof}

\begin{remark}\rm 
As in Section \ref{subsec:Discussion} we consider the special case when the window is of the form $r^{1/d}W$ for some fixed convex body $W\subset\RR^d$. For this, we fix  $j\in\{k,\ldots,d-1\}$ and use that $\log(1+x)$ behaves like $x$ for small values of $x$. Then one can easily check that, as $r\to\infty$,
$$\PP(F_j-\EE F_j\geq r)\leq \exp\big({-\boldsymbol{\Theta}(r^{1-k/d})}\big),$$ 
which is independent of $j$. This should be compared to the bound \eqref{eq:24-06-19} for the volume in this situation.
\end{remark}

\subsection*{Acknowledgement}

We would like to thank the other members of our team for stimulating discussions about the topic of this paper during our regular research seminars in the summer term 2019. We also thank Claudia Redenbach (Kaiserslautern) for providing the two simulations shown in Figure \ref{fig} and G\"unter Last (Karlsruhe) for encouraging us to study the case of expanding windows.\\
A.B.\ and A.G.\ were supported by the Deutsche Forschungsgemeinschaft (DFG) via RTG 2131 \textit{High- dimensional Phenomena in Probability -- Fluctuations and Discontinuity}. C.T.\ was supported by the DFG Scientific Network \textit{Cumulants, Concentration and Superconcentration}.


\addcontentsline{toc}{section}{References}

\begin{thebibliography}{30}\small

\bibitem{BormanTykesson}
Borman, E.I. and Tykesson, J.: Connectedness of Poisson cylinders in Euclidean space. Ann. Inst. H. Poincar\'e Probab. Statist. \textbf{52}, 102--126 (2016).

\bibitem{GieringerDissertation}
Gieringer, F.: Konzentrationsungleichungen f\"ur Poisson- und Binomialfunktionale in der Stochastischen Geometrie. Dissertation KIT, Karlsruhe (2016).

\bibitem{GieringerLast}
Gieringer, F. and Last, G.: Concentration inequalities for measures of a Boolean model. ALEA, Lat. Am. J. Probab. Math. Stat. 15, 151--166 (2018).

\bibitem{HeinrichBM}
Heinrich, L.: Large deviations of the empirical volume fraction for stationary Poisson grain models. Ann. Appl. Probab. \textbf{15}, 392--420 (2005).

\bibitem{HeinrichSpiessCLTVolume}
Heinrich, L. and Spiess, M.: Berry-Esseen bounds and Cram\'er-type large deviations for the volume distribution of Poisson cylinder processes. Lithuanian Math. J. \textbf{49}, 381--398 (2009).

\bibitem{HeinrichSpiessCLTVundS}
Heinrich, L. and Spiess, M.: Central limit theorems for volume and surface content of stationary Poisson cylinder processes in expanding domains. Adv. Appl. Probab. \textbf{45}, 312--331 (2013).

\bibitem{Hilario}
Hilario, M.R., Sidoravicius, V. and Teixeira, A.: Cylinders percolation in three dimensions. Probab. Theory Relat. Fields \textbf{163}, 613--642 (2015).

\bibitem{Hoffmann}
Hoffmann, L.M.: Mixed measures of convex cylinders and quermass densities of Boolean models. Acta Appl. Math. \textbf{105}, 141--156 (2009).

\bibitem{Houdre}
Houdr\'e, C.: Remarks on the deviation inequalities for functions of infinitely divisible random vectors. Ann. Probab. \textbf{30}, 1223--1237 (2002).

\bibitem{HugLastSchulte}
Hug, D., Last, G. and Schulte, M.: Second-order properties and central limit theorems for geometric functionals of Boolean models. Ann. Appl. Probab. \textbf{26}, 73--135 (2016).

\bibitem{LP}
Last, G. and Penrose, M.: {\em Lectures on the Poisson Process}. Cambridge University Press (2018).

\bibitem{Materon}
Matheron, G.: {\em Random Sets and Integral Geometry}. Wiley (1975).

\bibitem{Miles}
Miles, R.E.: A synopsis of Poisson flats in Euclidean spaces. In: {\em Stochastic Geometry. A Tribute to the Memory of Rollo Davidson}, edited by Harding, E.F. and Kendall D., Wiley (1974).

\bibitem{QiGuo}
Qi, F. and Guo, B.-N.: Complete monotonicities of functions involving the gamma and digamma functions. RGMIA Research Report Collection \textbf{7}, article 8 (2004).

\bibitem{Saulis}
Saulis, L. and Statulevi\v{c}ius, V.A.: {\em Limit Theorems for Large Deviations}. Kluwer Academic Publishers (1991).


\bibitem{SW}
Schneider, R. and Weil, W.: {\em Stochastic and Integral Geometry}. Springer (2008).

\bibitem{SpiessSpodarev}
Spiess, M. and Spodarev, E.: Anisotropic Poisson processes of cylinders. Methodol. Comput. Appl. Probab. \textbf{13}, 801--819 (2011).

\bibitem{TykessonWindisch}
Tykesson, J. and Windisch D.: Percolation in the vacant set of Poisson cylinders. Probab. Theory Related Fields \textbf{154}, 165--191 (2012).

\bibitem{Weil}
Weil, W.: Point processes of cylinders, particles and flats. Acta Appl. Math. \textbf{9}, 103--136 (1987).

\bibitem{Weil90}
Weil, W.: Iterations of translative integral formulae and non-isotropic Poisson processes of particles. Math. Z. \textbf{205}, 531--549 (1990).

\bibitem{Wu}
Wu, L.: A new modified logarithmic Sobolev inequality for Poisson point processes and several applications. Probab. Theory Relat. Fields \textbf{118}, 427--438 (2000).

\end{thebibliography}
\end{document}